\documentclass{amsart}
\usepackage{amsmath,amssymb,amsthm}
\usepackage{longtable}
\usepackage{a4wide}
\usepackage{epsf,stmaryrd,amscd} 
\usepackage{xspace}
\usepackage{graphicx,epstopdf}
\usepackage{latexsym,mathtools}
\usepackage{bbding} 

\newcommand{\cok}{{\mathrm{coker}\,}} 
\newcommand{\Aut}{{\mathrm{Aut}\,}} 

\newcommand{\SL}{{\mathrm{SL}\,}} 
\newcommand{\PSL}{{\mathrm{PSL}\,}} 
\newcommand{\PGL}{{\mathrm{PGL}\,}} 
\newcommand{\Sp}{{\mathrm{Sp}\,}} 
\newcommand{\SO}{{\mathrm{SO}\,}} 
\newcommand{\Spi}{{\mathrm{Spin}\,}}
\newcommand{\Ex}{{\mathrm{Ext}^1}}

\newcommand{\bF}{{\mathbb F}}

\newcommand{\bZ}{{\mathbb Z}} 

\newcommand{\mD}{{\mathfrak D}} 

\newcommand{\fA}{{\mathfrak A}}

\newcommand{\mA}{{\mathcal A}} 
\newcommand{\mP}{{\mathcal P}} 
\newcommand{\mQ}{{\mathcal Q}} 
\newcommand{\mX}{{\mathcal X}} 
\newcommand{\mY}{{\mathcal Y}} 
\newcommand{\mZ}{{\mathcal Z}}

\newcommand{\vx}{\ensuremath{\mathbf{x}}\xspace}

\newtheorem*{thmh}{Theorem}

\newtheorem{defn}{Definition}[section]
\newtheorem{theorem}[defn]{Theorem}

\newtheorem{lemma}[defn]{Lemma}
\newtheorem{cor}[defn]{Corollary}

\newtheorem{prop}[defn]{Proposition}

\title[Kac-Moody groups]{Presentations of affine Kac-Moody groups}
\author{Inna Capdeboscq}
\email{I.Capdeboscq@warwick.ac.uk}
\address{Department of Mathematics, University of Warwick, Coventry, CV4 7AL, UK}
\author{Karina Kirkina}
\email{kazina3@yahoo.com}
\address{Department of Mathematics, University of Warwick, Coventry, CV4 7AL, UK}
\author{Dmitriy Rumynin}
\email{D.Rumynin@warwick.ac.uk}
\address{Department of Mathematics, University of Warwick, Coventry, CV4 7AL, UK\newline
\hspace*{0.31cm}  Associated member of Laboratory of Algebraic Geometry, National
Research University Higher School of Economics, Russia}
\thanks{The research was partially supported by the Russian Academic Excellence Project `5--100' and by Leverhulme Foundation.}

\date{February 19, 2018}
\begin{document}
\begin{abstract}
How many generators and relations does $\SL_n(\bF_q[t, t^{-1}])$ need?
In this paper we exhibit its explicit presentation with 
$9$ generators and $44$ relations.

We investigate presentations of affine Kac-Moody
groups over finite fields.
Our goal is to derive finite presentations, 
independent of the field and with as few generators and relations as
we can achieve. It turns out that any simply connected 
affine Kac-Moody group over a finite field has a presentation
with at most 11 generators and 70 relations. We describe these
presentations
explicitly type by type. 
As a consequence,
we derive explicit presentations of Chevalley groups  $G(\bF_q[t, t^{-1}])$
and  explicit profinite  presentations of profinite Chevalley groups  $G(\bF_q[[t]])$. 

\end{abstract}

\maketitle

\section{Introduction}
\label{sc1}

Recently there has been a number of papers showing that various infinite families of groups 
have presentations with bounded number of generators and relations.
In particular,  the results of the following type were proved:

{\em
Let $\mA$ be a certain family of groups. There exists $C>0$ such that for any group $G\in\mA$, $G$ admits a presentation $\sigma(G)=\langle D_{\sigma(G)} \mid R_{\sigma(G)}\rangle$ such that
$$|D_{\sigma(G)}|+|R_{\sigma(G)}|<C.$$
}

This result is known if the family $\mA$ is a family of finite simple groups \cite{kn::GKKL1, kn::GKKL2, kn::GKKL3}, 
a family of affine Kac-Moody groups defined over finite fields \cite{kn::Cap}, 
and more recently   a family of Chevalley groups over various rings \cite{kn::CapLuRe}.
In fact, Guralnick, Kantor, Kassabov and Lubotzky
\cite{kn::GKKL1, kn::GKKL3} provide a  numerical bound on $C$ 
in the case when $\mA$ is a family of finite simple groups (with the possible exception of $^2\!G_2(3^a)$, $a\geq 1$): 

{\em 
If $G$ is a non-sporadic quasisimple finite group (with the possible exception of $^2\!G_2(3^a)$, $a\geq 1$),
then $G$ has  a presentation with at most $2$ generators and $51$ relations.
}

The goal of the current paper is to provide a quantitative statement in the case when $\mA$ is a family of affine Kac-Moody groups over finite fields and use this to derive numerical
bounds for some arithmetic groups defined over fields of positive characteristic.

\begin{theorem}
\label{theorem::main}
Let $G$ be a simply connected  affine Kac-Moody group of rank $n\geq 3$ defined over a finite field $\mathbb{F}_q$. 
If $q\geq 4$, $G$ has a presentation with  $2$ generators and at most  $72$ relations.

The result also holds if $q\in\{2,3\}$ provided that the Dynkin diagram of $G$ is not of type $\tilde{A}_2$ and does not contain a subdiagram of type $B_2$ or $G_2$ for $q=2$, 
and of type $G_2$ for $q=3$. If $G=\widetilde{A}_2(2)$ or $\widetilde{A}_2(3)$, $G$ has a presentation with at most $3$ generators and $29$ relations.
\end{theorem}

Our results are not restricted to simply connected groups. Using our techniques we derive  quantitative bounds  on the presentations of arbitrary affine Kac-Moody groups over finite fields. These results can be found in Section 6.

The upper bound of $72$ in Theorem~\ref{theorem::main} comes from the groups of type $\widetilde{C}_n^t$. In other types the bounds are better, as stated in the next theorem.

\begin{table}
\label{ta2}
\begin{center}
\caption{Generators and Relations of $\widetilde{X}(q)$}
\vskip 3mm
\bgroup
\def\arraystretch{1.3}
\begin{tabular}{|c|c|c|c||c|c|c|c|c|} \hline
Group & $|D_{\sigma}|$ & $|R_{\sigma}|$ & $|R_{\sigma}|$ & Group & $|D_{\sigma}|$ &
$|R_{\sigma}|$ & $|D_{\sigma}|$ &
$|R_{\sigma}|$   \\ \cline{3-4} \cline{6-9}
 &  & {$q$ odd} & {$q$ even}& & \multicolumn{2}{|c|}{$q$ odd} & \multicolumn{2}{|c|}{$q$ even} \\
\hline
$\widetilde{A}_2 (q)$ & 5 & 26  & 22 & $\widetilde{B}_3 (q)$, $\widetilde{B}_3^t(q)$ & 7 & 42  & 8 & 35 \\

$\widetilde{A}_3 (q)$  & 7 & {34} & {30}     & $\widetilde{B}_4(q)$, $\widetilde{B}^t_4(q)$ & 8 & 51  & 9 & 44  \\ 

$\widetilde{A}_n (q)$, $4\leq n\leq 7$  & 7 & 35  & 31 & $\widetilde{B}_n (q)$, $\widetilde{B}^t_n (q)$, $5\leq n \leq 8$ & 8 & 52  & 9 & 45 \\

$\widetilde{A}_n (q)$, $n\geq 8$  & 9 & 43  & 39  & $\widetilde{B}_n (q)$, $\widetilde{B}^t_n (q)$,  $n \geq 9$ & 9 & 56  & 10 & 49  \\

$\widetilde{D}_4 (q)$  & 7  & 38  & 34 & $\widetilde{C}_2 (q)$, $\widetilde{C}^{\prime}_2(q)$ & 7 & 49  & 9 & {39} \\

$\widetilde{D}_5 (q)$ & 7  & 39  & 35  &  $\widetilde{C}^t_2 (q)$  & 8  & 50  & 9 & {39}   \\

$\widetilde{D}_n (q)$, $6\leq n\leq 8$ & 7 & 38  & 34 & $\widetilde{C}_3 (q)$, $\widetilde{C}^{\prime}_3(q)$ & 8 & 58  & 10  & 48 \\

$\widetilde{D}_n (q)$, $n\geq 9$ & 8 & 42 & 38 & $\widetilde{C}^t_3 (q)$  & 9  & 59  & 10  & 48   \\

$\widetilde{E}_6 (q)$ &  7 & 36 & 32  & $\widetilde{C}_4 (q)$, $\widetilde{C}^{\prime}_4(q)$ & 9  & 64   & 11  & 54 \\

$\widetilde{E}_7 (q)$ &  6 & 30  & 26   & $\widetilde{C}^t_4 (q)$  & 10  & 65  & 11 & 54    \\

$\widetilde{E}_8 (q)$ & 7 & 34 & 30 & $\widetilde{C}_n (q)$,  $\widetilde{C}^{\prime}_n(q)$,   $5\leq n\leq 8$  & 9  & 65   & 11 & 55   \\

$\widetilde{G}_2(q)$, $\widetilde{G}^t_2(q)$ & 7 & 40 & 32 & $\widetilde{C}^{t}_n(q)$, $5 \leq n \leq 8$ & 10  & 66  & 11 & 55  \\

&  &  &  & $\widetilde{C}_n (q)$,  $\widetilde{C}^{\prime}_n(q)$,   $n\geq 9$  & 10 & 69   & 12 & {59}  \\
&  &  &  & $\widetilde{C}^{t}_n(q)$, $n \geq 9$ & 11 & 70 & 12 & {59}  \\
 &  &  &  & $\widetilde{F}_4 (q)$, $\widetilde{F}^t_4(q)$   &  8  &  50 &  9  & 43  \\
\hline
\end{tabular}
\egroup
\vskip 3mm
\end{center}
\end{table}

\begin{theorem}
\label{theorem::main2}
Let $G$ be a simply connected affine Kac-Moody group of rank $n\geq 3$ defined over a finite field $\mathbb{F}_q$. 
If $q\geq 4$, $G$ has a presentation $\sigma_G=\langle D_{\sigma}\mid R_{\sigma}\rangle$ where $|D_{\sigma}|$ and $|R_{\sigma}|$ are given in Table~1.
The result also holds if $q\in\{2,3\}$ provided that the Dynkin diagram of $G$ does not contain a subdiagram of type $B_2$ or $G_2$ for $q=2$, 
and of type $G_2$ for $q=3$.
\end{theorem}

Many mathematicians encounter  affine Kac-Moody groups defined over $\bF_q$ as Chevalley groups defined over $\bF_q[t,t^{-1}]$:
${\bf G}(\bF_q[t,t^{-1}])\cong \widetilde{\bf{G}}(\bF_q)/Z$ where  $Z\cong\bF_q^{\times}$ is a central subgroup of $\widetilde{\bf{G}}(\bF_q)$
 \cite[Section 2]{kn::MR}. For example, 
$\SL_n(\bF_q[t, t^{-1}])$  is the quotient of the simply connected
 affine Kac-Moody group  $\widetilde{A}_{n-1}(q)$ 
by its central subgroup $Z\cong \bF_q^{\times}$. 
Therefore we can obtain  a presentation of ${\bf G}(\bF_q[t,t^{-1}])$ from a presentation of $\widetilde{\bf{G}}(\bF_q)$ (as in Table 1) by adding one extra relation to kill a generator of $Z$.

The groups ${\bf G}(\bF_q[t,t^{-1}])$  can be generated by two elements (cf. Theorem~\ref{theorem::2_gen}). Therefore
we can change our presentation  to a 
presentation of ${\bf G}(\bF_q[t,t^{-1}])$ in these two generators. This change of generators costs two  extra relations (cf. Lemma 2.1).
The next theorem summarises this.

\begin{theorem}
\label{theorem::chevalley}
Let $\mathbf{G}$ be a simple simply connected Chevalley group scheme of rank $n\geq 2$. Take  $q=p^a$, $a\geq 1$ with $p$ a prime and  set $G=\mathbf{G}(\mathbb{F}_q[t, t^{-1}])$. Then $G$ has a presentation with  $2$ generators and at most $72$ relations with the possible exceptions of  $A_2(\mathbb{F}_2[t, t^{-1}])$,
 $B_n(\mathbb{F}_2[t, t^{-1}])$, $C_n(\mathbb{F}_2[t, t^{-1}])$, $G_2(\mathbb{F}_2[t, t^{-1}])$, $F_4(\mathbb{F}_2[t, t^{-1}])$,  $A_2(\mathbb{F}_3[t, t^{-1}])$ and $G_2(\mathbb{F}_3[t, t^{-1}])$. If $G=A_2(\mathbb{F}_2[t, t^{-1}])$ or  $A_2(\mathbb{F}_3[t, t^{-1}])$, $G$ has a presentation with at most $3$ generators and $30$ relations.
\end{theorem} 

The precise number of  generators and relations in a presentation of ${\bf G}(\bF_q[t,t^{-1}])$ can be  deduced from Table 1 by adding 1 relation, and 
for a presentation with $2$ generators can be
found in Table 2.

\begin{table}
\label{ta2a}
\begin{center}
\caption{Relations of $\mathbf{G}(\bF_q[t, t^{-1}])$ with 2 generators}
\vskip 3mm
\bgroup
\def\arraystretch{1.3}
\begin{tabular}{|c|c|c|c||c|c|c|c|} \hline
type & $\mathbf{G}$ & $|R_{\sigma}|$ & $|R_{\sigma}|$ & 
type & $\mathbf{G}$ & $|R_{\sigma}|$ & $|R_{\sigma}|$
\\ \cline{3-4} \cline{7-8}
 &  & {$q$ odd} & {$q$ even}& & & {$q$ odd} & {$q$ even} \\
\hline
 & $\SL_3$ & 29 & 25 &        & $\Spi_{7}$ & {45} & 38   \\
& $\SL_4$ & {37} & {33} &      & $\Spi_9$ & 54 & 47    \\ 
$A_{n-1}$ & $\SL_n$, $4 \leq n \leq 8$ & 38 & 34 &       $B_n$ & $\Spi_{2n+1}$, $5\leq n \leq 8$ & 55 & 48     \\ 
 & $\SL_n$, $n \geq 9$ & 46 & 42 &         & $\Spi_{2n+1}$, $n \geq 9$ & 59 & 52      \\           \hline
 & $\Spi_{8}$ & 41 & 37  &                             & $\Sp_4$  & 52 & {42}      \\      
 & $\Spi_{10}$ & 42 & 38   &                      & $\Sp_6$ & 61 & 51             \\
 $D_{n}$ & $\Spi_{2n}$, $6 \leq n \leq 8$ & 41 & 37  &    $C_n$  & $\Sp_8$  & 67 & 57       \\ 
  & $\Spi_{2n}$, $n \geq 9$ & 45 & 41  &        & $\Sp_{2n}$, $5 \leq n \leq 8$ & 68 & 58         \\ \cline{1-4}
  & $E_{6}$ & 39 &  35  &            & $\Sp_{2n}$, $n\geq 9$  & 72 & {62}     \\ \cline{5-8}
$E_n$ & $E_{7}$ & 33 & 29   &                 $F_4$ & $F_4$ & 53 & 46\\ \cline{5-8}
 & $E_8$ & 37 &  33  &                            $G_2$ & $G_{2}$ & 43 & 35\\
\hline
\end{tabular}
\egroup
\vskip 3mm
\end{center}
\end{table}


Capdeboscq, Lubotzky and Remy connect the presentations of Chevalley groups over 
$\mathbb{F}_q[t, t^{-1}]$ with the profinite presentations of Chevalley groups defined 
over $\mathbb{F}_q[[t]]$
\cite[Proposition 1.2]{kn::CapLuRe}. An immediate consequence of  their Proposition 1.2 combined with our Theorem~\ref{theorem::chevalley}, is the following statement.

\begin{theorem}
\label{theorem::profinite}
Let $\mathbf{G}$ be a simple simply connected Chevalley group scheme of rank at least $2$. For $q=p^a$, $a\geq 1$, $p$ a prime, consider a profinite group  
$G=\mathbf{G}(\mathbb{F}_q[[t]])$. 
Then $G$ has a profinite presentation with  $2$ generators and at most $72$ relations with the possible exceptions of $A_2(\mathbb{F}_2[[t]])$, $B_n(\mathbb{F}_2[[t]])$, $C_n(\mathbb{F}_2[[t]])$,
 $G_2(\mathbb{F}_2[[t]])$, $F_4(\mathbb{F}_2[[t]])$, $A_2(\mathbb{F}_3[[t]])$ and $G_2(\mathbb{F}_3[[t]])$. If $G=A_2(\mathbb{F}_2[[t]])$ or $A_2(\mathbb{F}_3[[t]])$, then $G$ has a profinite presentation with at most  $3$ generators and $31$ relations. 
\end{theorem}  

Let us describe the structure of the paper.
In  Section~\ref{sc2} we outline the proof providing background results.
We introduce simply connected Kac-Moody groups and show that
a Kac-Moody subgroup
of a 
simply connected Kac-Moody group
is a 
simply connected Kac-Moody group itself.
In Section~\ref{sc3} we examine the presentations of finite groups in
types $B_n$ and $D_n$.
In Section~\ref{sc4} we analyse the presentations of all untwisted Kac-Moody groups
giving details on the case-by-case basis.
We examine the twisted Kac-Moody groups
in Section~\ref{sc5}.
In Section~\ref{sc6}  we deal with the presentations of adjoint
Kac-Moody groups 
and classical groups over $\bF_q[t, t^{-1}]$.
We record all our findings in Tables~{4} and {5}.

$\bf{Acknowledgements.}$ The authors would like to thank Derek Holt for insightful information. 
The first and the third authors were partially supported by the Leverhulme Grant.
The first author is grateful to Oxford Mathematical Institute for its hospitality.

\section{Background results and outline of the proof}
\label{sc2}

\subsection{Presentations of finite quasisimple groups} 
Let us recall one of the main results of \cite{kn::GKKL3}. It asserts that every finite quasisimple group of Lie type admits a presentation 
with $2$ generators and at most $51$ relations. 
The authors  achieve this in two steps. First 
they give a presentation for  each family of groups of Lie type with very few generators and relations. 
Then, since every finite quasisimple group is
$2$-generated (see \cite[Theorem B]{kn::AG}), 
they use a reduction lemma  (Lemma 2.3 of \cite{kn::GKKL3}) to obtain the main result.
We now restate this reduction lemma to include the case of infinite groups.

\begin{lemma}[Lemma 2.3 of \cite{kn::GKKL3}]
Let $\sigma = \langle X\mid   R\rangle$ be a finite presentation of a group $G$,
$\pi: F\langle X \rangle \to G$ the corresponding natural map from a free group. 
If $D$ is a finite subset of $G$ such that 
$G= \langle D\rangle$,
then $G$ also has a presentation 
$\langle D \mid R'\rangle$ such that $|R'| = |D| + |R| - |\pi(X) \cap D|$.
\end{lemma}
\begin{proof}
Let $D=\{d_1, d_2, ... , d_l\}$.
The new presentation is obtained by a sequence of Tietze
transformations.
First we add generators:
$$
\sigma' = \langle X\cup D \mid   R\cup 
\{ d_i = \delta_i (X)\mid d_i \in D\setminus \pi(X) \} \cup \{ d_j = x_{j^\ast}\mid d_j \in D\cap \pi(X) \} \rangle
$$
where the second set in the union contains
one relation for each $d_i \in D\setminus \pi(X)$ 
expressing $d_i$ as a word $\delta_i(X)$ in $X$,
while the third set in the union contains
one relation for each $d_j \in D\cup \pi(X)$ 
where $d_j = \pi ( x_{j^\ast} )$ for some $x_j^*\in X$. 
We then  remove $X$:
$$
\sigma^{''} = \langle D \mid   R\!\mid_{x_j=\chi_j(D)} \cup \{ d_i = 
\delta_i (\chi_1(D),\ldots , \chi_k(D))\ \mbox{for}\ d_i\in D\setminus \pi(X) \} \rangle
$$
where each $x_j \in X$ is expressed as a word $\chi_j(D)$ in $D$ and $R\!\mid_{x_j=\chi_j(D)}$ is a result of this substitution in the relations  $R$.
\end{proof}

This result implies that, if a $2$-generated  group  has a presentation with $n$ generators and $m$ relations, it also has a presentation with $2$ generators and $m+2$ relations.

We need some of the presentations of quasisimple groups of Lie type obtained in \cite{kn::GKKL3}. We exhibit them in  Table~{3} 
(cf. Table 1 of \cite[ p. 93]{kn::GKKL3}).

\begin{table}
\label{ta1}
\begin{center}
\caption{Presentations of $\mathbf{G}(\bF_q)$ \cite{kn::GKKL3}}
\vskip 3mm
\bgroup
\def\arraystretch{1.3}
\begin{tabular}{|c|c|c|c|c|c|c|c|c|} \hline
Group &    \multicolumn{4}{|c|}{$q$ odd} & \multicolumn{4}{|c|}{$q$ even} \\ \cline{2-5}\cline{6-9}
& $|D_{\sigma}|$ &  $|R_{\sigma}|$ & label & contains & $|D_{\rho}|$& $|R_{\rho}|$ & label & contains  \\ 
\hline
$\text{SL}(2,q)$ & 3 & 9  & $\sigma_1$  & & 3 & 5 & $\rho_1$ &  \\
$\text{SL}(3,q)$  & 4 & 14  & $\sigma_2$   & $\sigma_1$ & 4 & 10 & $\rho_2$ & $\rho_1$  \\
$\text{SL}(4,q)$ & 5 & 20 & $\sigma_3$ & $\sigma_1$ & 5&16 & $\rho_3$ &  $\rho_1$ \\ 
$\text{SL}(4,q)$ & 6 & 21 & $\sigma_4$ & $\sigma_1$, $\sigma_2$ & 6 & 17 & $\rho_4$ &  $\rho_1$, $\rho_2$ \\ 
$\text{SL}(n, q)$, $5\leq n \leq 8$ & 5 & 21 & $\sigma_5$ & $\sigma_1$ & 5 & 17 & $\rho_5$ & $\rho_1$ \\
$\text{SL}(n, q)$, $5\leq n \leq 8$ & 6 & 22 & $\sigma_6$ & $\sigma_1$, $\sigma_2$ & 6 & 18 & $\rho_6$ & $\rho_1$, $\rho_2$ \\
$\text{SL}(n, q)$, $n \geq 9$ & 6 & 25 & $\sigma_7$ & $\sigma_1$ & 6 & 21 & $\rho_7$ & $\rho_1$ \\
$\text{SL}(n, q)$, $n \geq 9$ & 7 & 26 & $\sigma_8$ & $\sigma_1$, $\sigma_2$ & 7 & 22 & $\rho_8$ & $\rho_1$, $\rho_2$ \\
$\text{Sp}(4,q)$ & 5 & 27 & $\sigma_9$ & $\sigma_1$ (short) & & & & \\
$\text{Sp}(4,q)$ & 6 & 28 & $\sigma_{10}$ & $\sigma_1$ (twice) & 6 & 20 & $\rho_{10}$ & $\rho_1$ (twice) \\
$\text{Sp}(6,q)$, $\text{Spin}(7, q)$ & 6 & 36 &  & & 7 & 29& & \\
$\text{Sp}(8,q)$, $\text{Spin}(9, q)$ & 7 & 42 &  & & 8 & 35 & & \\
$\text{Sp}(2n,q)$, $5\leq n \leq 8$ & 7 & 43 &  & & 8 & 36 &  & \\
$\text{Spin}(2n+1, q)$, $5\leq n \leq 8$ & 7 & 43 & & & & & & \\
$\text{Sp}(2n,q)$, $n \geq 9$ & 8 & 47 & $\sigma_{11}$ & & 9 & 40 & $\rho_{11}$ & \\
$\text{Spin}(2n+1, q)$, $n \geq 9$ & 8 & 47 & $\sigma_{12}$ & & & & & \\
$\text{Spin}(8,q)$ & 6 & 29 &  &  & 6 & 25 &  &  \\
$\text{Spin}(2n,q)$, $5 \leq n \leq 8$ & 6 & 30 &  & & 6 & 26 & &  \\
$\text{Spin}(2n,q)$, $n \geq 9$ & 7 & 34 & $\sigma_{13}$ & $\sigma_1$ & 7 & 30 & $\rho_{13}$ & $\rho_1$ \\
$\text{G}_2 (q)$ & 6 & 31 & $\sigma_{14}$ &  $\sigma_1$ (twice)  & 6 & 23 & $\rho_{14}$ & $\rho_1$ (twice) \\
\hline
\end{tabular}
\egroup
\vskip 3mm
\end{center}
\end{table}

\begin{defn}\label{defn::sigma}
Let $B$ be a group and $A$ its subgroup. Suppose further that $B$ has a finite presentation $\sigma_B=\langle X_B\mid R_B\rangle$ and $A$ has a presentation 
$\sigma_A=\langle X_A\mid R_A\rangle$ such that $X_A\subset X_B$ and $R_A\subseteq R_B$. Then we say that $$\sigma_A\subseteq \sigma_B.$$
\end{defn}

Let us use this definition to explain the connection between
presentations 
of quasisimple groups of Lie type in Table~{3}.
We call presentations $\sigma_i$ if $q$ is odd. 
Theorem 4.5 of \cite{kn::GKKL3} gives  a presentation $\sigma_1$ of $\SL(2, q)$. 
Consider a group $G=\SL(3,q)$. If $\{\alpha_1, \alpha_2\}$ is the set of simple roots of $\SL(3,q)$, 
then $G$ contains a subgroup $L=\langle  X_{\alpha_1},  X_{-\alpha_1}\rangle\cong \SL(2,q)$, where $X_{\alpha_1}$ and  $X_{-\alpha_1}$ are the root subgroups of $\SL(3,q)$ (cf. \cite{kn::Car2}).
Theorem 5.1 of \cite{kn::GKKL3}  gives a presentation $\sigma_2$ of  $\SL(3, q)$
that  contains  $\sigma_L=\sigma_1$.

Theorem 6.1 of \cite{kn::GKKL3} gives  presentations   $\sigma_3$, $\sigma_4$ of  $\SL(4, q)$, $\sigma_5$, $\sigma_6$ of $\SL(n, q)$ for $5\leq n\leq 8$, and $\sigma_7$, $\sigma_8$ of  $\SL(n, q)$ for $n\geq 9$. 
If $\{\alpha_1, \alpha_2, .... , \alpha_{n-1}\}$ is the set of simple roots of $\SL(n,q)$, the group
$\SL(n,q)$ contains a subgroup  $M=\langle X_{\alpha_1},  X_{-\alpha_1}, X_{\alpha_2}, X_{-\alpha_2}\rangle\cong  \SL_3(q)$. From the proof of Theorem 6.1 it follows that
  $\sigma_2\subseteq \sigma_i$ for $i=4,6$ and $8$ where $\sigma_2$ is the presentation of $M$. This proof also shows that the shorter presentations $\sigma_3$, $\sigma_5$ and $\sigma_7$ contain $\sigma_1$. Thus we obtain that
  $$\sigma_1 \subseteq \sigma_i\ \mbox{for}\ i=3,5,7, \ \ \ 
\sigma_1\subseteq \sigma_2\subseteq \sigma_i\ \mbox{for}\ i=4,6,8.$$

Theorem  7.1 and Remark 7.4 of \cite{kn::GKKL3} give 
presentations $\sigma_9$ and $\sigma_{10}$ 
of  $\text{Sp}(4, q)$ . 
Let $\alpha_2$ be  a short root, $\alpha_1$ a long root. 
Both presentations contain a presentation $\sigma_1$ of its short-root
subgroup 
$L_2 =\langle  X_{\alpha_2},  X_{-\alpha_2}\rangle\cong\SL(2, q)$.
Besides   $\sigma_{10}$ contains a presentation  $\sigma_1$ of its long-root
subgroup 
$L_1 =\langle  X_{\alpha_1},  X_{-\alpha_1}\rangle$:
$$\sigma_1(L_1) \subseteq\sigma_{10}, \ \ \ 
\sigma_1(L_2) \subseteq \sigma_i\ \ \mbox{for }\ \ i=9,10.$$

We also need a presentation of the family of groups $\Spi(2n,q)$, $n\geq 4$, $q=p^a$. The result of \cite{kn::GKKL3} says that such groups have presentations with
$9$ generators and $42$ relations.
We shorten this estimate to $7$ generators and $34$ relations in Section~3.
We also give a slightly shorter presentation of
$\Spi(2n+1,q)$ in Section~3.

If $q$ is even, all the presentations get shorter: some of the relations are
no longer necessary. 
In one case an extra generator is required. By $\rho_i$ we denote the 
presentation in characteristic 2
corresponding to the presentation $\sigma_i$ in odd characteristic.

%

\subsection{Simply connected Kac-Moody groups and their presentations}
Recall that Kac-Moody groups over arbitrary fields were defined by
Tits \cite{kn::T}.  We are only interested in the case when the  group is split and the field of the definition
$\mathbb{F}=\bF_q$  is a finite field of $q=p^a$ elements ($a\geq 1$ and $p$ a prime).

Let $A=(A_{ij})_{n\times n}$ be a generalised Cartan matrix,
$\mD=(I, A, \mX,\mY,\Pi,\Pi^\vee )$  a root datum of type $A$.
Recall that this means 
\begin{itemize}
\item $I=\{1, 2, ... , n\}$,
\item $\mY$ is a free finitely generated abelian group,
\item $\mX=\mY^\ast =\hom (\mY,\bZ)$  is its dual  group,
\item $\Pi=\{\alpha_1, \ldots \alpha_n\}$ is a set of simple roots,
where $\alpha_i \in \mX$,
\item $\Pi^\vee =\{\alpha^\vee_1, \ldots \alpha^\vee_n\}$ is a set of  
 simple coroots,
where $\alpha^\vee_i \in \mY$,
\item for all $i,j\in I$,  
$\alpha_i (\alpha^\vee_j ) = A_{ij}$.
\end{itemize}

From Tits' definition an explicit presentation of these groups  was derived by Carter,  
a presentation by the explicit set of generators and relations 
\`{a} la Steinberg (cf. \cite{kn::Car} p. 224). 
The presentation depends on the field and root datum, so the resulting
group
can be denoted $G_\mD(\bF)$.
If $A$ is not of finite type (i.e., $A$ is not  a  Cartan matrix from classical Lie theory), the standard presentation is infinite:
the number of generators and the number of relations are both infinite.

Recall that $G_\mD$ is {\em simply connected} if $\Pi^\vee$ is a basis
of $\mY$ 
(see Section~\ref{sc6} for an example of a simply connected
affine group),
and $G_\mD(\bF)$ is $2$-$\it{spherical}$ if  $A$ is $2$-spherical. 
The latter means that for each $J\subseteq I$ with $|J|=2$, the submatrix 
$A_J:=(A_{ij})_{i,j\in J}$ is a classical Cartan matrix. Now,  for
each $\alpha\in\Pi\cup-\Pi$, 
let $X_{\alpha}$ be a root subgroup of 
$G_\mD(\bF)$. Then $X_{\alpha}\cong (\mathbb{F}_q, +)$,  and  for all $i,j\in I$ with $i\neq j$, set 
$$
L_i\coloneqq\langle X_{\alpha_i}\cup X_{-\alpha_i}\rangle
\ \ \mbox{and}\ \ 
L_{ij}\coloneqq\langle L_i \cup L_j\rangle=L_{ji}
.$$

 Abramenko and Muhlherr  \cite{kn::AM} (see also \cite[Th. 3.7]{Ca})
 proved a significant new  theorem about the presentations of a large class of Kac-Moody groups. 
In particular, they showed that those groups were finitely presented.

\begin{thmh} (Abramenko, Muhlherr)
Let $A$ be a $2$-spherical generalised Cartan matrix and ${\mathfrak D}$  a simply-connected root datum corresponding to $A$.
Suppose that the field $\mathbb{F}$ is finite and the following condition holds:
$$ L_{ij}/Z(L_{ij})\not\cong B_2(2), G_2(2), G_2(3), {^2\!F}_4(2)\ \ \mbox{for\ all} \ i,j\in I \ \ \ (*).$$
Let $\widetilde{G}$ be the direct limit of the inductive system formed by the $L_i$ and $L_{ij}$ for $i,j\in I$, with the natural inclusions.
Then the canonical homomorphism
 $\widetilde{G}\rightarrow G_\mD(\bF)$
is an isomorphism.
\end{thmh}

For an insightful description of $2$-spherical Kac-Moody groups we refer the reader to
the paper of Caprace ~\cite{Ca}. 

The following observation is useful.

\begin{prop}
\label{prop::sc}
Let $A$ be a $2$-spherical generalised Cartan matrix and  ${\mathfrak D}$ a simply connected root datum corresponding to $A$.
 Suppose that the field $\mathbb{F}$ is finite
 and  condition $(*)$ holds.
Let $J\subseteq I$ and $$L_ J=\langle L_i \mid i\in J\rangle.$$

Then $L_ J$ is a simply connected Kac-Moody group 
$G_{\mD( J)}(\bF)$ with a root datum of type $A_{ J}=(A_{ij})_{i,j\in J}$.
\end{prop}

\begin{proof}
It is clear once you construct the root datum for 
the $L_ J$.
If  $ J\subseteq I$, 
let  $\Theta= \{\alpha_i \mid i\in J\}$. 
Then
$ \Theta^\vee = \{\alpha^\vee_i \mid i\in J\}$ 
is the corresponding subset of $\Pi^\vee$.
Since ${\mathfrak D}$ is simply connected,
the coroot lattice splits into a direct sum
$$
Y = \langle\Theta^\vee\rangle \oplus \langle\Pi^\vee \setminus \Theta^\vee\rangle. 
$$
Hence, the root datum for $L_ J$ is 
$$
{\mathfrak D}(J) =
(J, A_{ J},  X/ \langle \Theta^\vee\rangle^\perp , \langle\Theta^\vee\rangle
, \Theta, \Theta^\vee )
.$$ Thus, ${\mathfrak D}( J)$ is simply connected. 
\end{proof}

A direct consequence of the result of Abramenko and Muhlherr 
is that $G_\mD(\mathbb{F}_q)$ has a presentation whose  generators
are generators of root subgroups $X_i$'s, $i\in I$  and with the relations that are defining relations of all $L_i$'s and $L_{ij}$'s with $i,j\in I, i\neq j$.
This allowed Capdeboscq \cite[Theorem 2.1]{kn::Cap} to show that the affine Kac-Moody groups (of rank at least $3$)  defined over finite fields had bounded presentations.

\begin{thmh} (Capdeboscq)
There exists $C>0$ such  that if $\widetilde{X}(q)$ is an affine Kac-Moody group of rank at least $3$ defined over a field $\mathbb{F}_q$ with $q=p^a\geq 4$ ($p$ a prime), 
then $\widetilde{X}(q)$ has a presentation $\sigma_{\widetilde{X}(q)}=\langle D_{\widetilde{X}(q)}\mid R_{\widetilde{X}(q)}\rangle$ such that $$|D_{\widetilde{X}(q)}|+|R_{\widetilde{X}(q)}|<C.$$ 
\end{thmh}
 This result is based on the following observation (see the proof of Theorem 2.1 of \cite{kn::Cap})  which is  a  direct consequence of the result of Abramenko and Muhlherr.

%

\begin{prop}
\label{prop::main}
Let $\widetilde{X}(q)$ be a  simply connected  $2$-spherical split Kac-Moody group of rank  $n\geq 3$ over the field of $q$ elements. Let $\Delta$ be the Dynkin diagram of $\widetilde{X}(q)$  with the vertices
labelled by $\beta_1, ... , \beta_n$. Suppose further that if $\Delta$ contains a subdiagram of type $B_2$, then $q\geq 3$, and if it contains a subdiagram of type $G_2$, $q\geq 4$.

Suppose that $\Delta$ contains $k$ proper  subdiagrams $\Delta_1$, $\Delta_2$, ...,  $\Delta_k$  such that  $\Delta=\cup_{i=1}^k\Delta_i$, and each pair of vertices of $\Delta$ is contained in $\Delta_i$ for some $i\in\{1, ... , k\}$. 
Let  $X_i(q):=\langle L_j\mid \alpha_j\in\Delta_i\rangle$. 
If $\sigma_{X_i(q)}=\langle D_i\mid R_i\rangle$ is a presentation of $X_i(q)$, then $\widetilde{X}(q)$ has a presentation 
$$\sigma_{\widetilde{X}(q)}=\langle D_1\cup D_2\cup \ldots \cup D_k \mid R_1\cup \ldots \cup R_k\cup \bigcup_{1\leq i<j\leq k} R_{ij}\rangle$$
where $R_{ij}$ are the relations coming from identifying the generators of $X_{ij}(q):=\langle L_k\mid \alpha_k\in \Delta_i\cap\Delta_j\rangle$ in $X_i(q)$ and in $X_j(q)$.
\end{prop}

Suppose now  that 
$\Delta=\Delta_1\cup\Delta_2$ on the level of vertices, that
is, $\Delta$ may contain edges, not included in $\Delta_1$ or $\Delta_2$.
We cannot use Proposition~\ref{prop::main} immediately because it is
not true that any pair of vertices is contained in $\Delta_i$.
To remedy it we may introduce $\Delta_3$ into the picture. Let $\Delta_3$ be a subdiagram of $\Delta$ based on vertices of  $(\Delta_1\setminus\Delta_2)\cup (\Delta_2\setminus\Delta_1)$:
now any pair of vertices of $\Delta$ is contained in  some $\Delta_i$. Proposition~\ref{prop::main}
gives
a presentation
$$\sigma_{\widetilde{X}(q)}=
\langle D_1\cup D_2\cup D_3 \mid R_1\cup R_2 \cup R_3 \cup R_{12} \cup
R_{13} \cup R_{23}\rangle .$$
Let $X^1_{3}(q)$ be the subgroup  of $\widetilde{X}(q)$ corresponding to
$\Delta_1\setminus\Delta_2$, and 
$\sigma_{X^1_{3}(q)}= \langle D^1_{3}\mid
R^1_{3}\rangle$
its presentation.
{\em Mutatis mutandis}, $\sigma_{X^2_{3}(q)}= \langle D^2_{3}\mid
R^2_{3}\rangle$ for $\Delta_2\setminus\Delta_1$.
Now we can choose a special presentation of $X_{3}(q)$:
$$
\sigma_{{X_3(q)}}=
\langle D^1_3\cup D^2_3 \mid R^1_3 \cup
R^2_3 \cup R^{**}_{3}\rangle $$
where  $R^{**}_{3}$ are those relations that include generators in
both  $D^1_3$ and $D^2_3$. Using 
$\sigma_{{X_3(q)}}$ in $\sigma_{\widetilde{X}(q)}$ allows us to
eliminate generators in $D_3$ using Tietze transformations.
Since $\Delta_1 \supseteq \Delta_1\setminus\Delta_2$  and $\Delta_2 \supseteq \Delta_2\setminus\Delta_1$, many relations
become superfluous as summarised in the next corollary.

\begin{cor}
\label{cor::main}
Let $\widetilde{X}(q)$ be a  simply connected  $2$-spherical split Kac-Moody group of rank  $n\geq 3$ over the field of $q$ elements. Let $\Delta$ be the Dynkin diagram of $\widetilde{X}(q)$  with the vertices
labelled by $\beta_1, ... , \beta_n$. Suppose further that if $\Delta$ contains a subdiagram of type $B_2$, then $q\geq 3$, and if it contains a subdiagram of type $G_2$, $q\geq 4$.


Suppose that $\Delta$ contains  three proper  
subdiagrams $\Delta_1$,
$\Delta_2$ and 
$\Delta_3$   
such that  $\beta_1, ... , \beta_n\in\Delta_1\cup\Delta_2$ and $\Delta_3$ is a subdiagram of $\Delta$ based on vertices of $(\Delta_1\setminus\Delta_2)\cup (\Delta_2\setminus\Delta_1)$.
If for $i=1,2$, $\sigma_{X_i(q)}=\langle D_i\mid R_i\rangle$ is a
presentation of $X_i(q)$
and $\sigma_{X_3(q)}$ is as described before this corollary, then $\widetilde{X}(q)$ has a presentation
$$
\sigma_{\widetilde{X}}= \langle D_1\cup D_2 \mid R_1\cup R_2\cup
R^*_3\cup R_{12}\rangle
$$
where the relations in $R^*_3$ are obtained from the
relations in $R^{**}_3$ by substituting generators in $D_3$ with
their expressions via generators of $D_1$ or $D_2$.
\end{cor}


\subsection{$2$-Generation}
A non-abelian finite simple group
$G$   can be generated  by two elements \cite[Theorem B]{kn::AG}. We are interested in 
 $h(G)$, the largest non-negative integer such that $G^{h(G)}$ can be generated by two elements.  Mar\'{o}ti and   Tamburini 
 have proved that  $h(G)> 2\sqrt{|G|}$
 \cite[Theorem 1.1]{kn::MT}. 
 
 Consider  a finite group $H=G_1^{m_1}\times ... \times G_t^{m_t}$
 where each $G_i$ is a non-abelian simple group
 and $G_i$ and $G_j$ are not isomorphic for all $i\neq j$. 
 It is known that a
subset of $H$ generates $H$ if and only if its projection into $G_i^{m_i}$  generates
$G_i^{m_i}$ for each $i$ \cite[Lemma 5]{kn::KL}.
We combine all this information in the following statement.

\begin{prop}
\label{prop::2_gen}
Let $G=G_1^a\times G_2^b$ where $G_1$ and $G_2$ are  finite non-abelian non-isomorphic quasisimple groups and $a$ and $b$ are non-negative integers with $0\leq a,b\leq 3$. Then $G$ is $2$-generated.
\end{prop}
\begin{proof} The smallest order of a finite non-abelian simple group is $60$. Hence, for any finite simple group $h(G)\geq 16$. The result now
follows immediately from the combination of  \cite[Theorem B]{kn::AG}, \cite[Theorem 1.1]{kn::MT} and \cite[Lemma 5]{kn::KL} and the fact that a quasisimple finite group is a Frattini cover of a finite simple group.
\end{proof}

Now suppose that $\widetilde{X}(q)$ is a Kac-Moody group over the field $\bF_q$ (where $q=p^a$, $a\geq 1$, $p$ a prime).
We will need the following result  proved in \cite{kn::Cap2} for large enough $q$ and clarified  in \cite{kn::CR}  to include the small values
of $p$ and $q$.

\begin{theorem}
\label{theorem::2_gen}
Let $\widetilde{X}(q)$  be a simply connected  affine Kac-Moody group of rank $n\geq 3$ defined over a finite field $\mathbb{F}_q$.
Then  $\widetilde{X}(q)$ is generated by $2$ elements, with the possible exceptions of $\widetilde{A}_2(2)$ and $\widetilde{A}_2(3)$ in which case it is generated by at most $3$ elements.
\end{theorem}

\subsection{}  In sections \ref{sc4} and \ref{sc5} we are going to use Proposition~\ref{prop::main} and its corollary to obtain bounded presentations of  Kac-Moody groups over finite fields.
We do this by looking at affine groups case-by-case. 
We then record our findings in Table~{1}. 
There are 7 infinite series and 7 exceptional types of affine Dynkin diagrams of rank at least 3
that are listed in the next sections. 
We use Dynkin labels for affine diagrams \cite{kn::Car2}.

\section{Presentation of $\Spi(n,q)$}
\label{sc3}
\subsection{Presentation of $\Spi(2n,q)$}
To obtain a presentation of $G=\Spi(2n,q)$, $n\geq 4$,  we use Corollary~\ref{cor::main}.

\begin{center}
\includegraphics{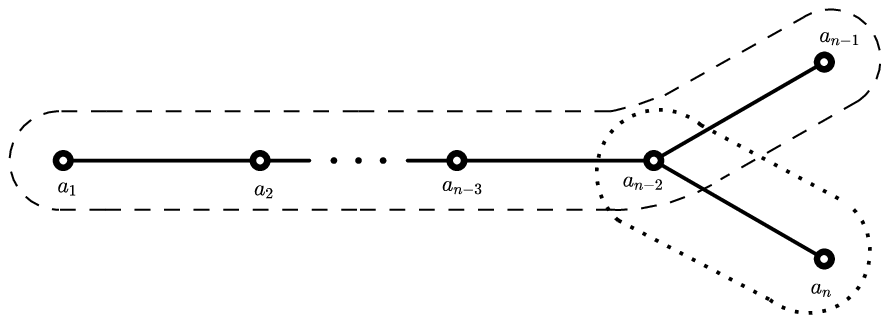}
\end{center}


By Proposition~\ref{prop::sc}, the groups $X_1(q)$, $X_2(q)$ and $X_3(q)$ are simply connected.
Let $\Delta_1$ be the subdiagram  of $\Delta$ whose  vertices  are the $n-1$ nodes $a_1, \ldots, a_{n-1}$. It has type $A_{n-1}$ and so $X_1(q)\cong \SL(n,q)$.
Let $\Delta_2$  be the subdiagram  of $\Delta$ whose vertices are the nodes $a_{n-2}$ and $a_n$. It is  of type $A_2$  and so $X_2(q)\cong \SL(3,q)$. 
Then $\Delta_3$  is the subdiagram  of $\Delta$ based on all vertices but $a_{n-2}$,  thus of type $A_{n-3}\times A_1\times A_1$.
Hence, $X_3(q)\cong \SL(n-2,q)\times \SL(2,q)\times \SL(2,q)$.
Clearly, $\Delta=\cup_{i=1}^2\Delta_i$. 
Therefore $G$ has a presentation  $$\sigma_G=\langle D_1\cup D_2\mid R_1\cup R_2\cup R_3^*\cup R_{12}\rangle$$  
as described in Corollary~\ref{cor::main}. 

Take a presentation $\sigma_{X_1(q)}$. If $n\geq 9$, $\sigma_{X_1(q)}=\sigma_7$, if $5\leq n\leq 8$, $\sigma_{X_1(q)}=\sigma_5$, and if
$n=4$, $\sigma_{X_1(q)}=\sigma_3$.
Consider a subgroup $X=L_{n-2}$ of $G$. 
Its Dynkin diagram is of type $A_1$ and so by Proposition~\ref{prop::sc}, $X\cong \SL(2,q)$.
From Table~{3} we know that $X$ has a presentation $\sigma_X=\langle D_X\mid R_X\rangle =\sigma_1$ with
 $|D_X|=3$ and $|R_X|=9$.
 Now $X\leq X_i(q)$ for $i=1,2$.  
The group $X_2(q)$ has a presentation $\sigma_{X_2(q)}=\sigma_2$ with 
 $\sigma_X\subseteq\sigma_{X_2(q)}$.  
 Since $X\leq X_1(q)$,  obviously,  $D_X\subseteq X_1(q)$. Thus  elements of $D_X$ can be expressed in terms of elements of $D_1$.
Moreover, the relations  $R_X$ hold,  as they hold in $X_1(q)$.
We use Tietze transformations to eliminate $D_X$,  $R_X$ and $R_{12}$ to obtain: 
 $$\sigma'_G=\langle D_1\cup (D_2\setminus D_X) \mid R_1\cup (R_2\setminus R_X)\cup R_3^*\rangle.$$

If $n\geq 9$,  we use $\sigma_7$ 
that requires $6$ generators and $25$
relations,  
so  that 
$|D_1\cup (D_2\setminus D_X)|=6+(4-3)=7$ and $|R_1\cup (R_2\setminus R_X)|=25+(14-9)=30$.

Consider $X_3(q)=(X_3(q)\cap X_1(q))\times (X_3(q)\cap X_2(q))$ where $X_3(q)\cap X_1(q)\cong \SL(n-2,q)\times \SL(2,q)$ and 
$X_3(q)\cap X_2(q)\cong \SL(2,q)$. Each factor  has two  generators (Proposition~\ref{prop::2_gen}).
Denote them by $a_1, a_2$ and $b_1, b_2$ respectively. Then $R^*_3=\{[a_1,b_1]=[a_1,b_2]=[a_2, b_1]=[a_2,b_2]=1\}$ and $|R^*_3|=4$.

Therefore  if $n\geq 9$, $G$ has  a presentation with $7$ generators and $34$ relations.
We call it $\sigma_{13}$.
If $n\leq 8$, we get shorter presentations (see Table~3).
Note that $\sigma_1 \subseteq \sigma_{13}$.

If $q$ is even,
the corresponding presentations 
have 4 fewer relations each. For instance, for $\rho_{13}$ the corresponding
calculation is
$|R_1\cup (R_2\setminus R_X) \cup R_3^* |=21+(10-5) + 4=30$.
This is typical for simply-laced groups: all of them will get 4 fewer relations in
characteristic 2. 

\subsection{Presentation of $\Spi(2n+1,q)$}
To obtain a presentation of $G=\Spi(2n +1,q)$, $n\geq 3$,  we use Corollary~\ref{cor::main}. 

\begin{center}
\includegraphics{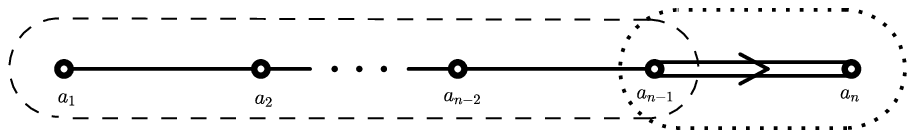}
\end{center}

By Proposition~\ref{prop::sc}, the groups $X_1(q)$, $X_2(q)$ and $X_3(q)$ are simply connected.
Let $\Delta_1$ be the subdiagram  of $\Delta$ whose  vertices  are the $n-1$ nodes $a_1, \ldots, a_{n-1}$. It has type $A_{n-1}$ and so $X_1(q)\cong \SL(n,q)$.
Let $\Delta_2$  be the subdiagram  of $\Delta$ whose vertices are the nodes $a_{n-1}$ and $a_n$. It is  of type $C_2$  and so $X_2(q)\cong \Sp(4,q)$. 
Then $\Delta_3$  is the subdiagram  of $\Delta$ based on all vertices but $a_{n-1}$,  thus of type $A_{n-2}\times A_1$.
Hence, $X_3(q)\cong \SL(n-1,q)\times \SL(2,q)$.
Clearly, $\Delta=\cup_{i=1}^2\Delta_i$. 
Therefore $G$ has a presentation
$$\sigma_G=\langle D_1\cup D_2\mid R_1\cup R_2\cup R_3^*\cup R_{12}\rangle$$  
as described in Corollary~\ref{cor::main}. 
Take presentations $\sigma_{X_1(q)}=\sigma_7$
and $\sigma_{X_2(q)}=\sigma_{9}$.
Consider a subgroup $X=L_{n-1}$ of $G$. 
Its Dynkin diagram is of type $A_1$ and so by Proposition~\ref{prop::sc}, $X\cong \SL(2,q)$.
From Table~{3} we know that $X$ has a presentation $\sigma_X=\langle D_X\mid R_X\rangle =\sigma_1$ with
 $|D_X|=3$ and $|R_X|=9$.
Observe that 
 $\sigma_X\subseteq\sigma_{X_1(q)}$.  
Since $X\leq X_2(q)$,  obviously,  $D_X\subseteq X_2(q)$.
Thus,  elements of $D_X$ can be expressed in terms of elements of $D_2$.
Moreover, the relations  $R_X$ hold,  as they hold in $X_2(q)$.
We use Tietze transformations to eliminate $D_X$,  $R_X$ and $R_{12}$ to obtain: 
 $$\sigma'_G=\langle (D_1\setminus D_X) \cup D_2 \mid (R_1 \setminus R_X) \cup R_2\cup R_3^*\rangle.$$
Observe that
$|(D_1\setminus D_X) \cup D_2|=(6-3)+5=8$ and $|(R_1\setminus R_X) \cup R_2)|=(25-9)+27=43$.

Consider $X_3(q)=(X_3(q)\cap X_1(q))\times (X_3(q)\cap X_2(q))$ where $X_3(q)\cap X_1(q)\cong \SL(n-1,q)$ and 
$X_3(q)\cap X_2(q)\cong \SL(2,q)$. Each factor  has two  generators (Proposition~\ref{prop::2_gen}).
Denote them by $a_1, a_2$ and $b_1, b_2$ respectively. Then $R^*_3=\{[a_1,b_1]=[a_1,b_2]=[a_2, b_1]=[a_2,b_2]=1\}$ and $|R^*_3|=4$.

We call this presentation $\sigma_{12}$.
Notice that the same method gives a presentation $\sigma_{11}$ of $\Sp (2n,q)$
of the same size. 
If $n\leq 8$,
we use a shorter presentation of $\SL (n,q)$ to obtain a shorter presentation
of $\Spi (2n+1,q)$ (or $\Sp (2n,q)$). We do not give them names but record their lenghts in Table~3.
If $q$ is even,
$\Spi (2n+1,q)\cong \Sp (2n,q)$
and
we call the corresponding presentation $\rho_{11}$.

\section{Untwisted affine Kac-Moody groups}
\label{sc4}
We now go through the calculations for the
4 infinite series and 5 exceptional  types of affine untwisted simply connected Kac-Moody groups of rank at least 3.

\subsection{$\widetilde{A}_n (q), n\geq 4$}
The Dynkin diagram $\Delta$ of $G=\widetilde{A}_n (q)$ consists of $n+1$ vertices: 

\begin{center}
\includegraphics{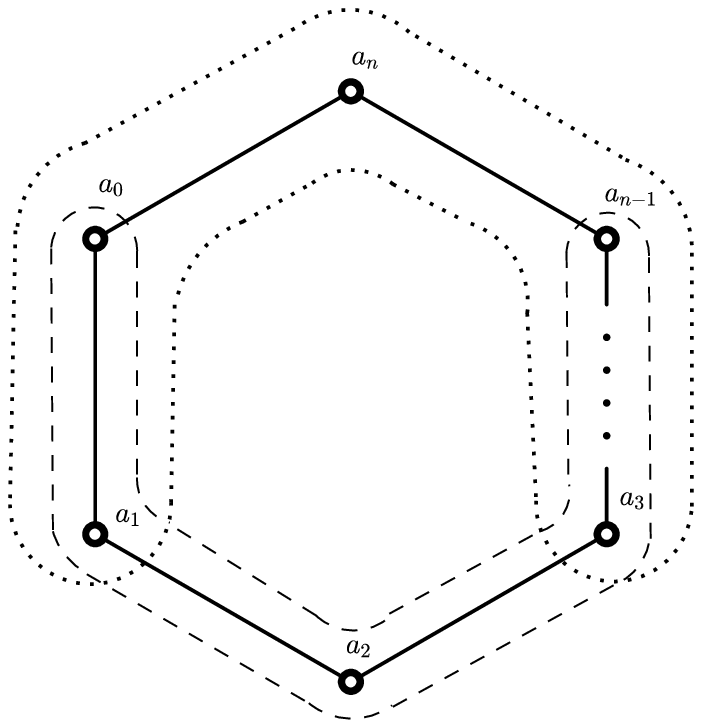}
\end{center}
We now use Corollary~\ref{cor::main}.
Notice, that by Proposition~\ref{prop::sc}, the groups $X_1(q)$, $X_2(q)$ and $X_3(q)$ are simply connected.
Let $\Delta_1$ be the subdiagram  of $\Delta$ whose  vertices  are the $n$ nodes $a_0$, $a_1$, $a_2, \ldots, a_{n-1}$. It has type $A_n$ and so $X_1(q)\cong \SL(n+1,q)$.
Let $\Delta_2$  be the subdiagram  of $\Delta$ whose vertices are the $n$ nodes $a_1$, $a_0$, $a_n, a_{n-1} \ldots, a_3$. It is  of type $A_n$  and so $X_2(q)\cong \SL(n+1,q)$. Then $\Delta_3$  is the subdiagram  of $\Delta$ whose vertices  are the two nodes $a_2$ and  $a_n$. This  diagram is of type $A_1\times A_1$ and thus corresponds to a subgroup $X_3(q)\cong \SL(2,q)\times \SL(2,q)$. 
Clearly, $\Delta=\cup_{i=1}^2\Delta_i$. 
Therefore $G$ has a presentation  $$\sigma_G=\langle D_1\cup D_2\mid R_1\cup R_2\cup R_3^*\cup R_{12}\rangle$$ 
as described in Corollary~\ref{cor::main}. 

Let us find presentations of $X_i(q)$ for $i=1,2,3$.
Consider a subgroup $X$ of $G$  that is generated by $L_0$ and $L_1$. 
Its Dynkin diagram is of type $A_2$ and so by Proposition~\ref{prop::sc}, $X\cong \SL(3,q)$.
From Table~{3} we know that $X$ has a presentation $\sigma_X=\langle D_X\mid R_X\rangle= \sigma_2$ with
 $4$ generators and $14$ relations.
 Now $X\leq X_i(q)$ for $i=1,2$. 
 Moreover,  $X_1(q)$ has a presentation $\sigma_{X_1(q)}=\sigma_7$  that requires $6$ generators and $25$ relations.
 Now $X_2(q)$ has a presentation $\sigma_{X_2(q)}=\sigma_8$
 with  $\sigma_X\subseteq\sigma_{X_2(q)}$.
 Since $X\leq X_1(q)$, $D_X\subseteq X_1(q)$ and relations $R_X$  already hold.
 We use Tietze transformations to eliminate $D_X$ and  $R_X$:
 $$\sigma_G'=\langle D_1\cup (D_2\setminus D_X)\mid R_1\cup (R_2\setminus R_X)\cup R_3^*\cup R_{12}\rangle.$$
 Now $\sigma_8$ requires $7$ generators and $26$ relations, and hence
$|D_1\cup (D_2\setminus D_X)|=6+(7-4)=9$ and $|R_1\cup (R_2\setminus R_X)|=25+(26-14)=37$.

Consider $X_3(q)=L_2\times L_n$.   
Since $L_2\cong L_n\cong \SL(2,q)$, each factor has $2$ generators (Proposition~\ref{prop::2_gen}). 
Denote them by $a_1, a_2$ and $b_1, b_2$ respectively. Then $R^*_3:=\{[a_1,b_1]=[a_1,b_2]=[a_2,b_1]=[a_2,b_2]=1\}$ and so $|R^*_3|=4$.

Finally, $\Delta_1\cap\Delta_2$ is of type $A_2\times A_{n-3}$. The corresponding group $\SL(3,q)\times \SL(n-2,q)$ has $2$ generators (Proposition~\ref{prop::2_gen}). 
We call them $c_1,d_1$ as elements of $X_1(q)$ and $c_2, d_2$ as elements of $X_2(q)$. Then $R_{12}=\{c_1=c_2, d_1=d_2\}$ and so $|R_{12}|=2$.

It follows that $G$ has  a presentation with $9$ generators and $37+4+2=43$ relations.
If $4\leq n\leq 7$, we may use $\sigma_5$ and $\sigma_6$ instead of $\sigma_7$ and $\sigma_8$. Then $G$ has a presentation with $7$ generators and $35$ relations. 

If $q$ is even, we use $\rho_2$, $\rho_7$ and $\rho_8$ instead
to get $9$ generators and $|R_1\cup (R_2\setminus R_X) \cup R_3^*  \cup R_{12}|=21+ (22-10) + 4 + 2 =39$ relations for $n \geq 8$, and $7$ generators and $31$ relations for $4\leq n\leq 7$.


\subsection{$\widetilde{A}_2 (q)$} To obtain a presentation of $G=\widetilde{A}_2 (q)$, we start with Proposition~\ref{prop::main}.

\begin{center}
\includegraphics{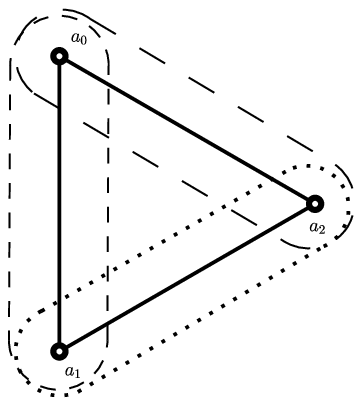}
\end{center}

Take $\Delta_1=A_2$ based on $a_0$ and $a_1$, $\Delta_2=A_2$ based on $a_0$ and $a_2$,  and $\Delta_3=A_2$ based on $a_1$ and $a_2$. Clearly, $\Delta=\cup_{i=1}^3\Delta_i$ and every pair of vertices is contained in at least  one of $\Delta_i$. 
Thus $G$ has a presentation $$\sigma_G=\langle D_1\cup D_2\cup D_3 \mid R_1\cup R_2\cup R_3\cup R_{12}\cup R_{13}\cup R_{23}\rangle.$$
Since $X_i(q)\cong A_2(q)$,  $X_i(q)\cong \SL(3,q)$ by Proposition~\ref{prop::sc}, and $X_i(q)$ has a presentation $\sigma_{X_i(q)}=\sigma_2$ ($i=1,2,3$).
Take $X=L_0$. Then  for $i=1,2$, $X\leq X_i(q)$ and $\sigma_X\subseteq \sigma_{X_i(q)}$.
It follows that $|D_1\cup D_2|=4+4-3$,  $|R_1\cup R_2|=14+14-9=19$ and the relations $R_{12}$ are superfluous. 

Take now $Y=L_1$. 
Without loss of generality we may assume that $\sigma_Y\subseteq\sigma_{X_3(q)}$ with $\sigma_Y=\sigma_1$. Then by Theorem 5.1 of \cite{kn::GKKL3}, 
$\sigma_{X_3(q)}=\langle D_Y\cup\{c\}\mid R_Y\cup R_3^c\rangle$ where  $D_Y=\{u,t,h\}$ (in the notations of 
Theorem 5.1 of \cite{kn::GKKL3})  and 
$R_3^c$ is a set of $5$ relations involving $c$ and elements of $D_Y$.
Since $Y\leq X_1(q)$,  obviously,  $D_Y\subseteq X_1(q)$. Thus $u$, $t$ and $h$  can be expressed in terms of elements of $D_1$.
This makes the relations $R_{13}$ superfluous.
Moreover, the relations  $R_Y$ hold,  as they hold in $X_1(q)$.
We use Tietze transformations to eliminate $D_Y$, $R_Y$ and $R_{13}$: 
$$\sigma'_G=\langle D_1\cup D_2\cup\{c\} \mid R_1\cup R_2\cup R_3^c\cup R_{23}\rangle.$$
Notice that $X_3(q)\cong \SL(3,q)$ is generated by its subgroups $L_1$ and $L_2$ (cf.  Lemma 2.1 of \cite{kn::Cap2}).
Hence, $c$ can be expressed in terms of elements of $D_1\cup D_2$.
Finally, since $\Delta_2\cap\Delta_3=A_1$ and  $A_1(q)$ is $2$-generated (Proposition~\ref{prop::2_gen}), $|R_{23}|=2$.
Therefore
$$\sigma^*_G=\langle D_1\cup D_2\ \mid R_1\cup R_2\cup R_3^c\cup R_{23}\rangle,$$
and so
$G$ has a presentation with $5$ generators  and $26$ relations. If $q$ is even,
we use $\rho_2$ instead and the corresponding calculation gives $5$ generators and
 $(10 + 10 - 5) + 5 + 2 = 22$ relations.

\subsection{$\widetilde{A}_3 (q)$}
To obtain a presentation of $G=\widetilde{A}_3 (q)$, we start with Proposition~\ref{prop::main}.
\begin{center}
\includegraphics{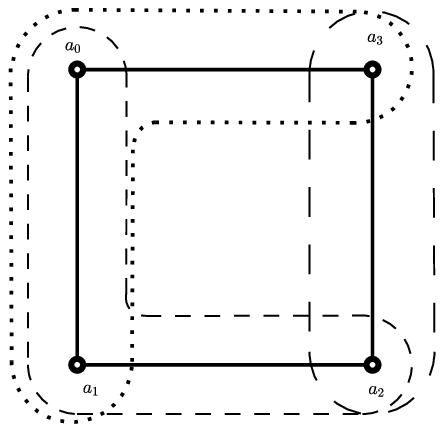}
\end{center}

Take $\Delta_1=A_3$ based on $a_0$, $a_1$ and $a_2$,
 $\Delta_2=A_3$ based on $a_0$, $a_1$ and $a_3$,  and $\Delta_3=A_2$ based on $a_2$ and $a_3$.
Thus $G$ has a presentation $$\sigma_G=\langle D_1\cup D_2\cup D_3 \mid R_1\cup R_2\cup R_3\cup R_{12}\cup R_{13}\cup R_{23}\rangle.$$
Since $X_i(q)\cong A_3(q)$,  $X_i(q)\cong \SL(4,q)$ by Proposition~\ref{prop::sc}  for $i=1,2$.
Now $X_1(q)$ has a presentation $\sigma_{X_1(q)}=\sigma_3$ with $5$ generators and $20$ relations.

Consider a subgroup $X$ of $G$  generated by $L_0$ and $L_1$. 
Its Dynkin diagram is of type $A_2$ and so by Proposition~\ref{prop::sc}, $X\cong \SL(3,q)$.
From Table~{3} we know that $X$ has a presentation $\sigma_X=\langle D_X\mid R_X\rangle= \sigma_2$ with
 $4$ generators and $14$ relations.
 Now $X\leq X_i(q)$ for $i=1,2$. 
 Moreover $X_2(q)$ has a presentation $\sigma_{X_2(q)}=\sigma_4$
 with  $\sigma_X\subseteq\sigma_{X_2(q)}$.
 Since $X\leq X_1(q)$, $D_X\subseteq X_1(q)$ and the relations $R_X$  already hold.
 We use Tietze transformations to eliminate $D_X$,  $R_X$ and $R_{12}$:
 $$\sigma_G'=\langle D_1\cup (D_2\setminus D_X)\cup D_3\mid R_1\cup (R_2\setminus R_X)\cup R_3\cup R_{13}\cup R_{23}\rangle.$$
 
 Since $\sigma_4$ requires $6$ generators and $21$ relations, 
$|D_1\cup (D_2\setminus D_X)|=5+(6-4)=7$ and $|R_1\cup (R_2\setminus R_X)|=20+(21-14)=27$.

Take now $Y=L_3$. 
Without loss of generality we may assume that $\sigma_Y\subseteq\sigma_{X_3(q)}$ with $\sigma_Y=\sigma_1$. 
Then by Theorem 5.1 of \cite{kn::GKKL3}, 
$\sigma_{X_3(q)}=\langle D_Y\cup\{c\}\mid R_Y\cup R_3^c\rangle$ where  $D_Y=\{u,t,h\}$ (in the notations of 
Theorem 5.1 of \cite{kn::GKKL3})  and 
$R_3^c$ is a set of $5$ relations involving $c$ and elements of $D_Y$.
Since $Y\leq X_2(q)$,  obviously,  $D_Y\subseteq X_2(q)$. Thus $u$, $t$ and $h$  can be expressed in terms of elements of $D_1\cup (D_2\setminus D_X)$.
This makes the relations $R_{23}$ superfluous.
Moreover, the relations  $R_Y$ hold,  as they hold in $X_2(q)$.
We use Tietze transformations to eliminate $D_Y$, $R_Y$ and $R_{23}$: 
$$\sigma''_G=\langle D_1\cup (D_2\setminus D_X)\cup\{c\} \mid R_1\cup (R_2\setminus D_X)\cup R_3^c\cup R_{13}\rangle.$$
Notice that $X_3(q)\cong \SL(3,q)$ is generated by its subgroups $L_2$ and $L_3$ (cf.  Lemma 2.1 of \cite{kn::Cap2}).
Hence, $c$ can be expressed in terms of elements of $D_1\cup (D_2\setminus D_X)$.
Finally, since $\Delta_1\cap\Delta_3=A_1$ and $A_1(q)$ is $2$-generated (Proposition~\ref{prop::2_gen}), $|R_{13}|=2$.
Therefore
$$\sigma^*_G=\langle D_1\cup (D_2\setminus D_X) \mid R_1\cup (R_2\setminus D_X)\cup R_3^c\cup R_{13}\rangle,$$
and so
$G$ has a presentation with $7$ generators  and $34$ relations. For $q$ even,
we use $\rho_2$, $\rho_3$ or $\rho_4$ instead to get $16+(17-10) + 5 + 2=30$ relations.

\subsection{$\widetilde{B}_n (q)$, $n\geq  9$} 
To obtain a presentation of $G=\widetilde{B}_n (q)$, $n\geq 9$,  we use Corollary~\ref{cor::main}.

\begin{center}
\includegraphics{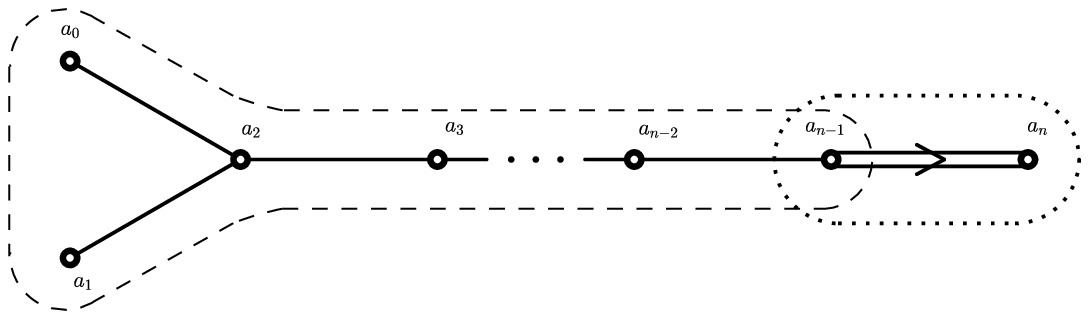}
\end{center}

By Proposition~\ref{prop::sc}, the groups $X_1(q)$, $X_2(q)$ and $X_3(q)$ are simply connected.
Let $\Delta_1$ be the subdiagram  of $\Delta$ whose  vertices  are the $n$ nodes $a_0, a_1,  \ldots, a_{n-1}$. It has type $D_n$ and so $X_1(q)\cong \Spi(2n,q)$.
Let $\Delta_2$  be the subdiagram  of $\Delta$ whose vertices are the nodes $a_{n-1}$ and $a_n$. It is  of type $C_2$  and so $X_2(q)\cong \Sp(4,q)$. 
Then $\Delta_3$  is the subdiagram  of $\Delta$ based on all vertices but $a_{n-1}$,  thus of type $D_{n-1}\times A_1$.
Hence, $X_3(q)\cong \Spi(2n-2,q)\times \SL(2,q)$.
Clearly, $\Delta=\cup_{i=1}^2\Delta_i$. 
Therefore $G$ has a presentation  $$\sigma_G=\langle D_1\cup D_2\mid R_1\cup R_2\cup R_3^*\cup R_{12}\rangle.$$  
as described in Corollary~\ref{cor::main}.

Consider a subgroup $X=L_{n-1}$ of $G$. 
Its Dynkin diagram is of type $A_1$ and so by Proposition~\ref{prop::sc}, $X\cong \SL(2,q)$.
From Table~{3} we know that $X$ has a presentation $\sigma_X=\langle D_X\mid R_X\rangle =\sigma_1$ with  $|D_X|=3$ and $|R_X|=9$. 
 Now $X\leq X_i(q)$ for $i=1,2$.  
 The group $X_2(q)$ has a presentation $\sigma_{X_2(q)}=\sigma_9$ if   $q$ is odd
 ($\sigma_{X_2(q)}=\rho_{10}$ if $q$ is even).
By the results of Section 3, $X_1(q)$ has a presentation $\sigma_{X_1(q)}=\sigma_{13}$ and  $\sigma_X\subseteq\sigma_{X_1(q)}$.

Since $X\leq X_2(q)$,  obviously,  $D_X\subseteq X_2(q)$ and  the relations  $R_X$ hold,  as they hold in $X_2(q)$.
We use Tietze transformations to eliminate $D_X$,  $R_X$ and $R_{12}$ to obtain: 
 $$\sigma'_G=\langle (D_1\setminus D_X)\cup D_2 \mid (R_1\setminus R_X)\cup R_2 \cup R_3^*\rangle.$$
 Now $\sigma_{13}$ requires $7$ generators and $34$ relations,  so  that 
 $|(D_1\setminus D_X)\cup D_2|=(7-3)+5=9$  and $|(R_1 \setminus R_X)\cup R_2| =(34-9) +27=52$  if $q$ is odd. The corresponding calculation for even $q$
uses $\rho_{10}$ and $\rho_{13}$ instead: 
$|(D_1\setminus D_X)\cup D_2|=(7-3)+6=10$  and $|(R_1 \setminus R_X)\cup R_2|=(30 -5 )+20=45$.

Consider $X_3(q)\cong \Spi(2n-2,q)\times \SL(2,q)$. Each factor  has two  generators (Proposition~\ref{prop::2_gen}).
Denote them by $a_1, a_2$ and $b_1, b_2$ respectively. Then $R^*_3=\{[a_i, b_j]=1, 1\leq i,j\leq 2\}$ and so $|R^*_3|=4$.

Therefore $G$ has  a presentation with  $9$ generators and $56$ relations if $q$ is odd, and $10$ generators and $49$ relations if $q$ is even.

\subsection{$\widetilde{B}_n(q),\ 3\leq n\leq 8$} To obtain a presentation of $G=\widetilde{B}_3 (q)$,  we use Corollary~\ref{cor::main} similar to the case $\widetilde{B}_n(q)$ with $n\geq 9$. 
However, in this case we take $\Delta_1=A_3$ and $\Delta_2=C_2$. Then $X_1(q)\cong \SL(4,q)$ has a presentation
$\sigma_{X_2(q)}=\sigma_3$ and $X_2(q)\cong \Sp (4, q)$ has a presentation $\sigma_{X_2(q)}=\sigma_{9}$.
Taking $X=L_2$ (as in the previous case), we have that $X\leq X_i(q)$ for $i=1,2$,
$X$ has a presentation $\sigma_X=\langle D_X\mid R_X\rangle=\sigma_1$ 
and $\sigma_X\subseteq \sigma_{X_1(q)}$.
This is possible because $L_1$ and $L_2$ are conjugate inside $X_1 (q)$. 

Since $X\leq X_2(q)$, $D_X\subseteq X_2(q)$ and $R_X$ hold, as they hold in $X_2(q)$.
Finally, $X_3(q)=(X_1(q)\cap X_3(q))\times (X_2(q)\cap X_3(q))=(\SL(2,q)\times \SL(2,q))\times \SL(2,q)$,
and as before (using Proposition~\ref{prop::2_gen}) we obtain that $|R_3^*|=4$.
Using Tietze transformations we obtain that  $G$ has a presentation 
$$\sigma_G'=\langle (D_1\setminus D_X)\cup D_2 \mid (R_1\setminus R_X) \cup R_2\cup R_3^*\rangle$$
with $(5-3)+5=7$ generators  and $(20-9)+27+4=42$ relations if $q$ is odd. If $q$ is even the corresponding calculation gives 
$(5-3) + 6 = 8$ generators and $(16-5)+ 20+4=35$ relations.

 To obtain a presentation of $G=\widetilde{B}_n (q)$ for $4\leq n\leq 8$,  we use Corollary~\ref{cor::main} just as in the case $\widetilde{B}_n(q)$ with $n\geq 9$. 
 If $G=\widetilde{B}_4(q)$, $\Delta_1$ is of type $D_4$,
 and so $X_1(q)\cong \Spi(8,q)$ and has a presentation
$\sigma_{X_1(q)}$ that is the reduced $\sigma_{13}$ with $6$ generators and $29$ relations. Thus  $G$ has a presentation with $(6-3)+5=8$ generators and $(29-9) +27+4=51$ relations if $q$ is odd. If $q$ is even we get $(6-3)+6=9$ generators and $(25 -5) +20+4=44$ relations.

 Finally, if $G=\widetilde{B}_n(q)$ with $5\leq n\leq 8$, $\Delta_1$ is of type $D_n$, and so $X_1(q)\cong \Spi(2n,q)$ and has a presentation
 $\sigma_{X_1(q)}$
 that is the reduced $\sigma_{13}$ with $6$ generators and $30$ relations. Thus  $G$ has a presentation with $(6-3)+5=8$ generators and $(30-9)+27+4=52$ relations if $q$ is odd. If $q$ is even we get
$(6-3)+6=9$ generators and $(26 -5) +20+4=45$ relations.

\subsection{$\widetilde{C}_n (q)$, $n\geq 3$}
To obtain a presentation of $G=\widetilde{C}_n (q)$, $n\geq 3$,  we use Corollary~\ref{cor::main}.

\begin{center}
\includegraphics{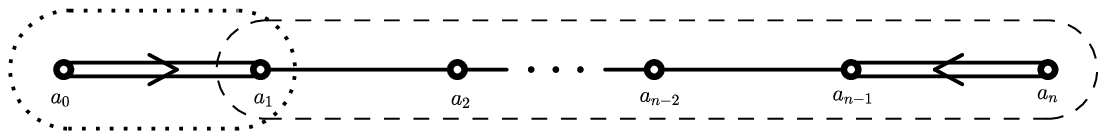}
\end{center}

By Proposition~\ref{prop::sc}, the groups $X_1(q)$, $X_2(q)$ and $X_3(q)$ are simply connected.
Let $\Delta_1$ be the subdiagram  of $\Delta$ whose  vertices  are the $n$ nodes $a_1,  \ldots, a_n$. It has type $C_n$ and so $X_1(q)\cong \Sp(2n,q)$.
Let $\Delta_2$  be the subdiagram  of $\Delta$ whose vertices are the nodes $a_0$ and $a_1$. It is  of type $C_2$  and so $X_2(q)\cong \Sp(4,q)$. 
Then $\Delta_3$  is the subdiagram  of $\Delta$ based on all vertices but $a_1$,  thus of type $A_1\times C_{n-1}$.
Hence, $X_3(q)\cong \SL(2,q)\times \Sp(2n-2,q)$.
Clearly, $\Delta=\Delta_1\cup\Delta_2$. 
Therefore $G$ has a presentation  $$\sigma_G=\langle D_1\cup D_2\mid R_1\cup R_2\cup R_3^*\cup R_{12}\rangle.$$  
as described in Corollary~\ref{cor::main}. 

Take a presentation $\sigma_{X_1(q)}=\sigma_{11}$. 
Consider a subgroup $X=L_1$ of $G$. 
Its Dynkin diagram is of type $A_1$ and so by Proposition~\ref{prop::sc}, $X\cong \SL(2,q)$.
From Table~{3} we know that $X$ has a presentation $\sigma_X=\langle D_X\mid R_X\rangle =\sigma_1$ with $|D_X|=3$ and $|R_X|=9$.
 Now $X\leq X_i(q)$ for $i=1,2$.  
The group $X_2(q)$ has a presentation   $\sigma_{X_2(q)}=\sigma_9$ (or $\sigma_{X_2(q)}=\rho_{10}$ if $q$ is even).
By Theorem 7.1 of \cite{kn::GKKL3},  $\sigma_X\subseteq\sigma_{X_2(q)}$.
 Since $X\leq X_1(q)$,  obviously,  $D_X\subseteq X_1(q)$. Thus  elements of $D_X$ can be expressed in terms of elements of $D_1$.
Moreover, the relations  $R_X$ hold,  as they hold in $X_1(q)$.
We use Tietze transformations to eliminate $D_X$,  $R_X$ and $R_{12}$ to obtain: 
 $$\sigma'_G=\langle D_1\cup (D_2\setminus D_X) \mid R_1\cup (R_2\setminus R_X)\cup R_3^*\rangle.$$
 Therefore
$|D_1\cup (D_2\setminus D_X)|=8+(5-3)=10$ and $|R_1\cup (R_2\setminus R_X)|=47+(27-9)=65$ if $q$ is odd. For even $q$ the corresponding
calculations are
$|D_1\cup (D_2\setminus D_X)|=9+(6-3)=12$ and $|R_1\cup (R_2\setminus R_X)|=40+(20-5)=55$.

Finally, consider $X_3(q)\cong \Sp(2n-2,q)\times \SL(2,q)$. Each factor has two  generators (Proposition~\ref{prop::2_gen}). Thus as in the previous case we obtain
$|R^*_3|=4$.

Therefore $G$ has  a presentation with  $10$ generators and $69$ relations if $q$ is odd, and $12$ generators and $59$ relations if $q$ is even. 
For $3 \leq n \leq 8$ we obtain shorter presentations, see Table~{1}.

\subsection{$\widetilde{C}_2 (q)$}
The only difference with the previous case is that $\Delta_1=C_2$, and thus $X_1(q)$ has a presentation $\sigma_{X_1(q)}=\sigma_9$ if $q$ is odd and 
$\sigma_{X_1(q)}=\rho_{10}$ if $q$ is even. 
Replacing $\sigma_{11}$ and $\rho_{11}$ by these, we obtain that
$G=\widetilde{C}_2(q)$ has a presentation
with
$5+(5-3)=7$ generators and $27+(27-9)+4=49$ relations if $q$ is odd,  
and $6+(6-3)=9$ generators and $20+(20-5)+4=39$ relations if $q$ is even.

\subsection{$\widetilde{D}_n (q), n\geq 6$}
Let us assume that $n\geq 9$.
We  use Corollary~\ref{cor::main}. 
Let $\Delta_1=A_{n-1}$ on vertices $a_1, a_2, ...  , a_{n-1}$, $\Delta_2=A_2\times A_2$ on vertices $a_0$, $a_2$, $a_{n-2}$ and $a_n$,  
 and $\Delta_3=(A_1)^4\times A_{n-5}$ on
all vertices but $a_2$ and $a_{n-2}$. 

\begin{center}
\includegraphics{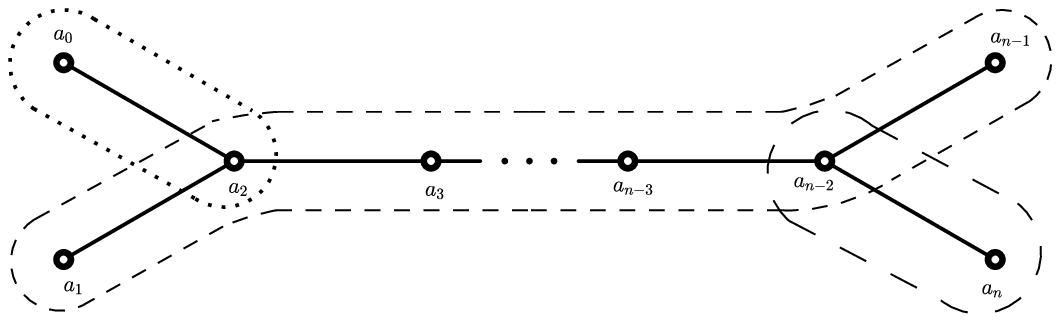}
\end{center}

By Proposition~\ref{prop::sc}, $X_i(q)$ are simply connected for $1\leq i\leq 3$, and so
$X_1(q)\cong \SL(n,q)$, $X_2(q)\cong \SL(3,q)\times \SL(3,q)$ and $X_3(q)\cong (\SL(2,q))^4\times \SL(n-4, q)$.
Thus $G$ has a presentation 
$$\sigma_G=\langle D_1\cup D_2 \mid R_1\cup R_2\cup R_3^*\cup R_{12}\rangle.$$
Now $X_1(q)$ has a presentation $\sigma_{X_1(q)}=\sigma_7$ with $6$ generators and $25$ relations. 

Consider a subgroup $X$ of $G$  generated by $L_2$ and $L_{n-2}$.
Its Dynkin diagram is of type $A_1\times A_1$ and so by Proposition~\ref{prop::sc}, $X\cong \SL(2,q)\times \SL(2,q)$.
From Table~{3} it follows that $X$ has a presentation 
$$\sigma_X=\langle D_X^2\cup D_X^{n-2}\mid R_X^2\cup R_X^{n-2}\cup R_X^*\rangle$$ where 
$\sigma_X^i=\langle D_X^i\mid R_X^i\rangle$  is a presentation of $L_i\cong \SL(2,q)$ for $i=2$ and  $n-2$,  $\sigma_X^i=\sigma_1$
and $R_X^*$ are the relations that ensure that $[L_2, L_{n-2}]=1$. Since both $L_2$ and $L_{n-2}$ are $2$-generated (Proposition~\ref{prop::2_gen}), we may chose 
$c_1, c_2\in L_2$ and $d_1, d_2\in L_{n-2}$ so that $R_X^*=\{[c_i, d_j]=1, 1\leq i,j\leq 2\}$. Hence,   $|R_X^*|=4$.

 Now $X\leq X_i(q)$ for $i=1,2$.  Moreover $X_2(q)$ has a presentation
  $\sigma_{X_2(q)}$
 with  $\sigma_X\subseteq\sigma_{X_2(q)}$. Indeed, take $\sigma_{X_2(q)}=\langle D_2\mid R_2\rangle$ such that
 $D_2=D_2^2\cup D_2^{n-2}$ and $R_2=R_2^2\cup R_2^{n-2}\cup R_2^*$ where for $i=2$ and $n-2$, 
 $\sigma_{X_2(q)}^i=\langle D_2^i\mid R_2^i\rangle$ is a presentation of a subgroup $L_{0,2}$ of $X_2(q)$  if $i=2$,  and $L_{n-2, n}$ if $i=n-2$, 
  $\sigma_{X_2(q)}^i=\langle D_2^i\mid R_2^i\rangle=\sigma_2$ and $R_2^*$ are the relations that ensure that $[L_{0,2}, L_{n-2, n}]=1$.
  Using \cite{kn::AG} and \cite[p.~745, Corollary]{kn::GK},  we may take $c_1', c_2'\in L_{0,2}$  and $d_1', d_2'\in L_{n-2, n}$ with $c_1'=c_1$, $d_1'=d_1$. Then let $R_2^*= R_X^*\cup \{[c_1', d_2']=[c_2', d_1']=[c_2', d_2']=1\}$.
 Then $R_X^*\subseteq R_2^*$ and   $|R_2^*\setminus R_X^*|=3$.

  Since $X\leq X_1(q)$, $D_X\subseteq X_1(q)$ and relations $R_X$  already hold in $X_1(q)$.
 We use Tietze transformations to eliminate $D_X$,  $R_X$ and $R_{12}$:
 $$\sigma_G'=\langle D_1\cup (D_2\setminus D_X)\mid R_1\cup (R_2\setminus R_X)\cup R^*_3\rangle.$$
  Notice that
$|D_1\cup (D_2\setminus D_X)|=6+(4-3)+(4-3)=8$ and $|R_1\cup (R_2\setminus R_X)|=25+(14-9)+(14-9)+3=38$.

Finally, $X_3(q)=(X_3(q)\cap X_1(q))\times (X_3(q)\cap X_2(q))$ where $X_3(q)\cap X_1(q)=L_1\times\langle L_3, ... , L_{n-3}\rangle\times L_{n-1}\cong \SL(2,q)\times \SL(n-4,q)\times \SL(2,q)$ and $X_3(q)\cap X_2(q)=L_0\times L_n\cong \SL(2,q)\times \SL(2,q)$. Since each of the factors requires only two generators (Proposition~\ref{prop::2_gen}), we obtain that $|R_3^*|=4$.
Therefore $G$ has a presentation with $8$ generators and $42$ relations if $q$ is odd. For even $q$ the corresponding calculation gives $8$ generators and 
$21 + ((10-5) + (10-5) +3) +4 = 38$ relations.

If $6\leq n\leq 8$, then
the above argument
works with little variation.
The subgroup 
$X_1(q)\cong \SL(n,q)$ has a  presentation
$\sigma_{X_1(q)}=\sigma_5$ with $5$ generators and $21$ relations. 
The rest of the argument does not change,  
producing a presentation of $G$ with $7$ generators and $38$ relations if $q$ is odd
and $7$ generators and $34$ relations if $q$ is even.

\subsection{$\widetilde{D}_n (q), n = 4,5$}
This time we use Corollary~\ref{cor::main} with $\Delta_1=D_n$ based on all vertices but $a_0$, and $\Delta_2=A_2$ based on vertices $a_0$ and $a_2$. 
Then $X_1(q)\cong \Spi(2n, q)$ and  $X_2(q)\cong \SL(3,q)$.
Thus $\Delta_3=A_1^4$ if $n=4$, and $\Delta_3=A_1^2\times A_3$ if $n=5$, giving
$X_3(q)\cong \SL(2,q)^4$ and $X_3(q)\cong \SL(2,q)^2\times \SL(4,q)$ respectively.

Consider a subgroup $X=L_2$ of $G$. Then $X\leq X_i(q)$ for $i=1,2$, $X\cong \SL(2,q)$ and $X$ has  a presentation $\sigma_X=\langle D_X\mid R_X\rangle=\sigma_1$.
Now $X_1(q)$ has  a presentation $\sigma_{X_1(q)}$ that is the reduced $\sigma_{13}$. 
The group $X_2(q)$ has  a presentation $\sigma_{X_2(q)}=\sigma_2$ and $\sigma_X\subseteq \sigma_{X_2(q)}$.
Since $X\leq X_1(q)$, $D_X\subseteq X_1(q)$ and $R_X$ hold as they hold in $X_1(q)$. We use Tietze transformations to eliminate $D_X$, $R_X$ and $R_{12}$ to obtain a presentation
$$\sigma_G=\langle D_1\cup (D_2\setminus D_X)\mid R_1\cup (R_2\setminus R_X)\cup R_3^*\rangle$$
where as usual $|R_3^*|=4$ (using Proposition~\ref{prop::2_gen}).
Thus $G$ has a presentation with $6+(4-3)=7$ generators and $29+(14-9)+4=38$ relations if $n=4$, and $6+(4-3)=7$ generators and $30+(14-9)+4=39$ relations if $n=5$. For even $q$ the corresponding calculations give $25 + (10 -5) +4 = 34$ relations if $n =4$ and $26 + (10-5) +4 = 35$ relations
if $n=5$.

\subsection{$\widetilde{E}_6 (q)$}
This time we use Corollary~\ref{cor::main}. Take $\Delta_1=A_5$ on vertices $a_1, a_2, a_3, a_5$ and $a_6$,  $\Delta_2=A_3$ on vertices $a_0$, $a_4$ and $a_3$. Then $\Delta_3=A_2\times A_2\times A_2$  is based on all vertices but  $a_3$. 

\begin{center}
\includegraphics{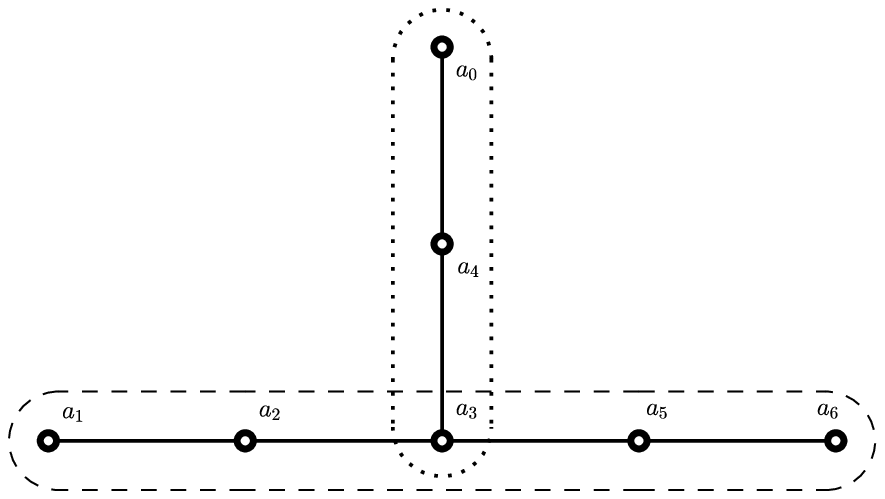}
\end{center}

 Hence, $G=\widetilde{E}_6(q)$ has a presentation
$$\sigma_G=\langle D_1\cup D_2\mid R_1\cup R_2\cup R_3^*\cup R_{12}\rangle.$$
Using Proposition~\ref{prop::sc} and Table~{3} we have that
$X_1\cong \SL(6,q)$  has a presentation $\sigma_{X_1(q)}=\sigma_5$,  and 
$X_2\cong \SL(4,q)$ has a presentation $\sigma_{X_2(q)}=\sigma_3$.

Consider a subgroup $X=L_3\cong \SL(2,q)$ of $G$. Then $X$ has a presentation $\sigma_X=\langle D_X\mid R_X\rangle=\sigma_1$. Now $X\leq X_i(q)$ for $i=1,2$, and  Theorem 6.1 of \cite{kn::GKKL3} implies that $\sigma_X\subseteq \sigma_{X_2(q)}$. Since $X\leq X_1(q)$, $D_X\subseteq X_1(q)$ and the relations $R_X$ hold (as they hold in $X_1(q)$). We use Tietze transformations to eliminate $D_X$, $R_X$ and $R_{12}$.
Hence, $G$ has a presentation $$\sigma'_G=\langle D_1\cup (D_2\setminus D_X) \mid R_1\cup (R_2\setminus R_X) \cup R_3^*\rangle.$$
Finally, $X_3(q)=(X_{3}(q)\cap X_1(q))\times (X_{3}(q)\cap X_2(q)) \cong (\SL(3,q)\times \SL(3,q))\times \SL(3,q)$. Both factors require $2$ generators (Proposition~\ref{prop::2_gen}),  implying $|R^*_3|=4$. Therefore
$G$ has a presentation with 
$5+(5-3)=7$ generators and $21+(20-9)+4=36$ relations if $q$ is odd. For even $q$ the corresponding calculation
gives $7$ generators and $17 + (16 -5) +4 = 32$ relations.

\subsection{$\widetilde{E}_7 (q)$}
Again we use Corollary~\ref{cor::main}. Take $\Delta_1=A_7$ based  on all vertices but  $a_5$,  and $\Delta_2=A_3$  based on vertices  $a_4$ and $a_5$. 
Then $\Delta_3=A_3\times A_3\times A_1$  is based on all vertices but  $a_4$. 

\begin{center}
\includegraphics{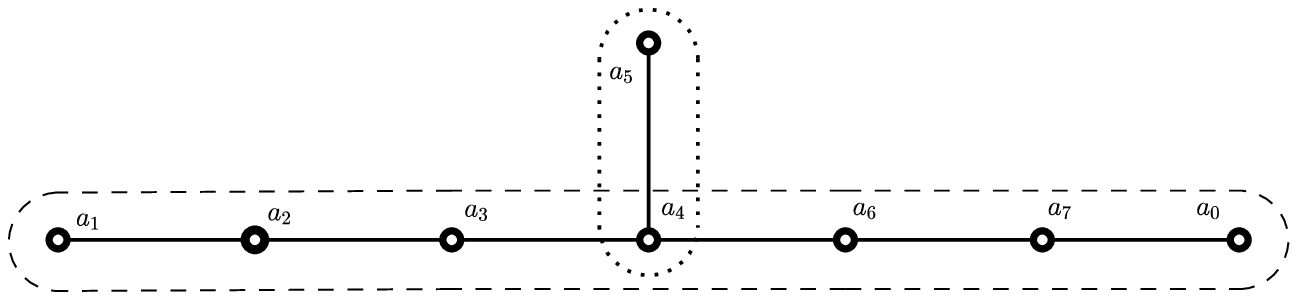}
\end{center}

Hence, $G=\widetilde{E}_7(q)$ has a presentation as described in Corollary~\ref{cor::main}.
Using Proposition~\ref{prop::sc} and Table~{3} we have that
$X_1\cong \SL(8,q)$  has a presentation $\sigma_{X_1(q)}=\sigma_5$,  
$X_2\cong \SL(3,q)$ has a presentation $\sigma_{X_2(q)}=\sigma_2$.
Consider a subgroup $X=L_4\cong \SL(2,q)$ of $G$. Then $X$ has a presentation $\sigma_X=\langle D_X\mid R_X\rangle=\sigma_1$. Now $X\leq X_i(q)$ for $i=1,2$, and   $\sigma_X\subseteq \sigma_{X_2(q)}$. Since $X\leq X_1(q)$, $D_X\subseteq X_1(q)$ and the relations $R_X$ hold (as they hold in $X_1(q)$). We use Tietze transformations to eliminate $D_X$, $R_X$ and $R_{12}$.
Hence, $G$ has a presentation $$\sigma'_G=\langle D_1\cup (D_2\setminus D_X) \mid R_1\cup (R_2\setminus R_X) \cup R_3^*\rangle.$$
Finally, $X_3(q)=(X_{3}(q)\cap X_1(q))\times (X_{3}(q)\cap X_2(q)) \cong (\SL(4,q)\times \SL(4,q))\times \SL(2,q)$. Both factors require $2$ generators (Proposition~\ref{prop::2_gen}),  implying $|R^*_3|=4$. Therefore
$G$ has a presentation with
$5+(4-3)=6$ generators and $21+(14-9)+4=30$ relations if $q$ is odd. If $q$ is even the corresponding calculation gives $6$ generators
and $17 + (10-5) +4 = 26$ relations.

\subsection{$\widetilde{E}_8 (q)$}
We use Corollary~\ref{cor::main}. Take $\Delta_1=A_8$ based  on all vertices but  $a_6$,  and $\Delta_2=A_2$  based on vertices  $a_5$ and $a_6$. 
Then $\Delta_3=A_5\times A_2\times A_1$  is based on all vertices but  $a_5$. 

\begin{center}
\includegraphics{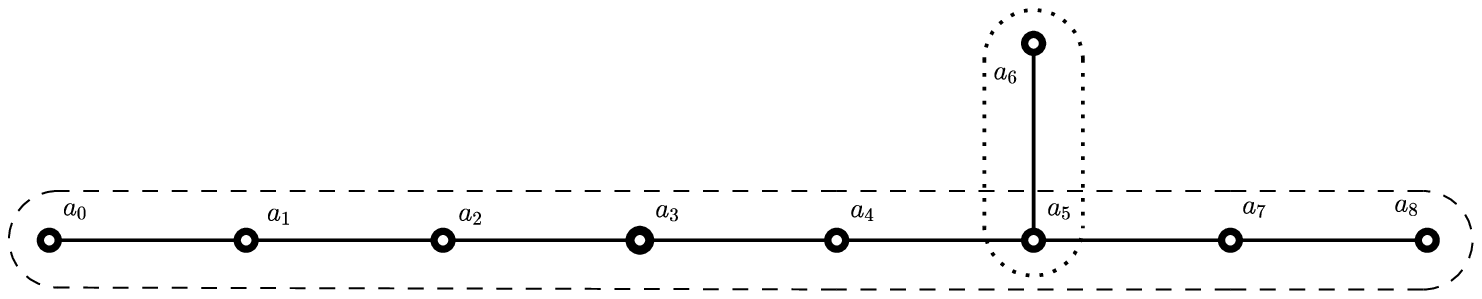}
\end{center}

Hence, $G=\widetilde{E}_8(q)$ has a presentation as described in Corollary~\ref{cor::main}.
Using Proposition~\ref{prop::sc} and Table~{3} we have that $X_1\cong \SL(9,q)$ has a presentation $\sigma_{X_1(q)}=\sigma_7$,  $X_2\cong \SL(3,q)$ has a presentation $\sigma_{X_2(q)}=\sigma_2$.
Consider a subgroup $X=L_5\cong \SL(2,q)$ of $G$. Then $X$ has a presentation $\sigma_X=\langle D_X\mid R_X\rangle=\sigma_1$. Now $X\leq X_i(q)$ for $i=1,2$, and   $\sigma_X\subseteq \sigma_{X_2(q)}$. Since $X\leq X_1(q)$, $D_X\subseteq X_1(q)$ and the relations $R_X$ hold (as they hold in $X_1(q)$). We use Tietze transformations to eliminate $D_X$, $R_X$ and $R_{12}$.
Hence, $G$ has a presentation $$\sigma'_G=\langle D_1\cup (D_2\setminus D_X) \mid R_1\cup (R_2\setminus R_X) \cup R_3^*\rangle.$$
Finally, $X_3(q)=(X_{3}(q)\cap X_1(q))\times (X_{3}(q)\cap X_2(q)) \cong (\SL(6,q)\times \SL(3,q))\times \SL(2,q)$. Both factors require $2$ generators (Proposition~\ref{prop::2_gen}),  implying $|R^*_3|=4$. Therefore
$G$ has a presentation with 
$6+(4-3)=7$ generators and $25+(14-9)+4=34$ relations if $q$ is odd. If $q$ is even we get $7$ generators
and $21 + (10-5) +4 =30$ relations.

\subsection{$\widetilde{F}_4 (q)$}


This time  we use Proposition~\ref{prop::main} with $k=5$.  
Take $\Delta_1=A_3$  based on vertices $a_0$, $a_1$ and $a_2$, 
$\Delta_2=C_2$ based on  vertices $a_2$ and $a_3$,  
$\Delta_3=A_2$ based on vertices  $a_3$ and $a_4$, 
$\Delta_4=A_2\times A_2$ based on  all vertices but  $a_2$, and finally
$\Delta_5=A_3\times A_1$ based on all vertices but  $a_3$.

\begin{center}
\includegraphics{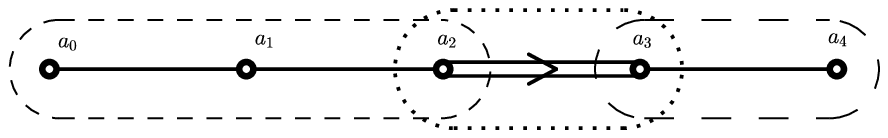}
\end{center}

Taking subgroups $X_i(q)$ corresponding to $\Delta_i$ for $1\leq i\leq 5$, we obtain a presentation of $G$ 
$$\sigma_G=\langle D_1\cup D_2\cup D_3\cup D_4\cup D_5\mid R_1\cup R_2\cup R_3\cup R_4\cup R_5\cup\bigcup_{ i<j} R_{ij}\rangle$$
as described in Proposition~\ref{prop::main}. Notice that $R_{13}=\emptyset$.

Using Proposition~\ref{prop::sc}  and Table 3 we have that
$X_1(q)\cong \SL(4,q)$ has a presentation $\sigma_{X_1(q)}=\sigma_3$, 
$X_2(q)\cong \Sp(4,q)$ has a presentation  $\sigma_{X_2(q)}=\sigma_9$ (or $\rho_{10}$ if $q$ is even),  and 
$X_3(q)\cong \SL(3,q)$ has a presentation $\sigma_{X_3(q)}=\sigma_2$.

Take $X=L_2\cong \SL(2,q)$. Then  $X\leq X_i(q)$ for $i=1,2$ and  Theorem 6.1 of \cite{kn::GKKL3} implies that $\sigma_X\subseteq \sigma_{X_1(q)}$. 
Since $X\leq X_2(q)$, $D_X\subseteq X_2(q)$ and $R_X$ hold as they hold in $X_2(q)$.
We use Tietze transformations to eliminate $D_X$, $R_X$ and $R_{12}$ to obtain a presentation

$$\sigma^{(1)}_G=\langle (D_1\setminus D_X)\cup D_2\cup D_3\cup D_4\cup D_5\mid (R_1\setminus R_X)\cup R_2\cup R_3\cup R_4\cup R_5\cup \bigcup_{ i<j} R_{ij}\setminus R_{12}\rangle.$$

Take $Y=L_3\cong \SL(2,q)$. Then  $Y\leq X_i(q)$ for $i=2,3$  and $\sigma_Y\subseteq \sigma_{X_3(q)}$. 
Again we use Tietze transformations. This time we eliminate $D_Y$, $R_Y$ and $R_{23}$ to obtain

$$\sigma^{(2)}_G=\langle (D_1\setminus D_X)\cup D_2\cup (D_3\setminus D_Y)\cup D_4\cup D_5\mid (R_1\setminus R_X)\cup R_2\cup (R_3\setminus R_Y)\cup R_4\cup R_5\cup\bigcup_{ i<j} R_{ij}\setminus (R_{12}\cup R_{23})\rangle.$$

By Proposition~\ref{prop::sc}, $X_4(q)=(X_4(q)\cap X_1(q))\times X_3(q)\cong \SL(3,q)\times \SL(3,q)$.
Each factor is $2$-generated (Proposition~\ref{prop::2_gen}). Let us denote these  pairs of generators by $c_1, c_2$ and $d_1, d_2$ respectively.
In fact, \cite[p.~745, Corollary]{kn::GK} allows us to choose $c_1\in L_0\leq X_4(q)\cap X_1(q) $ and $d_1\in L_4\leq X_3(q)$.
Then $X_4(q)$ has a presentation $\sigma_{X_4(q)}=\langle c_1, c_2, d_1, d_2\mid R_{c_1, c_2}\cup R_{d_1, d_2}\cup R_4^*\rangle$
where $\langle  c_1, c_2\mid R_{c_1, c_2}\rangle$ is a presentation of $X_4(q)\cap X_1(q)\cong \SL(3,q)$, 
 $\langle  d_1, d_2\mid R_{d_1, d_2}\rangle$ is a presentation of $X_4(q)\cap X_3(q)=X_3(q)\cong \SL(3,q)$, 
 and $R_4^*=\{[c_1,d_1]=[c_1,d_2]=[c_2,d_1]=[c_2,d_2]=1\}$. Since $X_4(q)\cap X_1(q)\leq X_1(q)$, $c_1, c_2\in X_1(q)$ and the relations
 $R_{c_1, c_2}$ hold as they hold in $X_1(q)$. Similarly, $d_1, d_2\in X_3(q)$ and relations
 $R_{d_1, d_2}$ hold as they hold in $X_3(q)$. We now use Tietze transformations to eliminate $c_1, c_2, d_1, d_2$, $R_{c_1, c_2}\cup R_{d_1, d_2}$ and
 $R_{14}\cup R_{34}$. 
 Now we may eliminate relations $R_{24}$: $R_{24}$ identify $X_2(q)\cap X_4(q)$. Note that $X_2(q)\cap X_4(q)=(X_2(q)\cap X_3(q))\cap (X_3(q)\cap X_4(q))$, and we have already identified $X_2(q)\cap X_3(q)$ and 
 $X_3(q)\cap X_4(q)$.
 Thus $G$ has a presentation
 $$\sigma^{(3)}_G=\langle (D_1\setminus D_X)\cup D_2\cup (D_3\setminus D_Y)\cup D_5\mid (R_1\setminus R_X)\cup R_2\cup (R_3\setminus R_Y)\cup R^*_4\cup R_5\cup \bigcup_{i=1}^5 R_{i5}\rangle.$$

Finally,  by Proposition~\ref{prop::sc}, $X_5(q)=X_1(q)\times (X_3(q)\cap X_5(q))\cong \SL(4,q)\times \SL(2,q)$.  
Each factor is $2$-generated (Proposition~\ref{prop::2_gen}). Let us denote these  pairs of generators by $c'_1, c'_2$ and $d'_1, d'_2$ respectively. Notice that  \cite[p.~745, Corollary]{kn::GK}  implies that we may choose $c_1'=c_1$ and $d_1'=d_1$. 
Then $X_5(q)$ has a presentation $\sigma_{X_5(q)}=\langle c_1, c'_2, d_1, d'_2\mid R_{c_1, c'_2}\cup R_{d_1, d'_2}\cup R_5^*\rangle$
where $\langle  c_1, c'_2\mid R_{c_1, c'_2}\rangle$ is a presentation of $X_5(q)\cap X_1(q)=X_1(q)\cong \SL(4,q)$, 
 $\langle  d_1, d'_2\mid R_{d_1, d'_2}\rangle$ is a presentation of $X_3(q)\cap X_5(q)\cong \SL(2,q)$, 
 and $R_5^*=\{[c_1,d_1]=[c_1,d'_2]=[c'_2,d_1]=[c'_2,d'_2]=1\}$. Since $X_5(q)\cap X_1(q)=X_1(q)$, $c_1, c'_2\in X_1(q)$ and relations
 $R_{c_1, c'_2}$ hold as they hold in $X_1(q)$. Similarly, $d_1, d'_2\in X_3(q)$ and relations
 $R_{d_1, d'_2}$ hold as they hold in $X_3(q)$. We use Tietze transformations to eliminate $c_1, c'_2, d_1, d'_2$, $R_{c_1, c'_2}\cup R_{d_1, d'_2}$ and
 $R_{15}\cup R_{35}$ to obtain
$$\sigma^{(4)}_G=\langle (D_1\setminus D_X)\cup D_2\cup (D_3\setminus D_Y)\mid (R_1\setminus R_X)\cup R_2\cup (R_3\setminus R_Y)\cup R^*_4\cup R^*_5\cup R_{25}\cup R_{45}\rangle.$$
 Now we may eliminate relations $R_{25}$: $R_{25}$ identify $X_2(q)\cap X_5(q)$. Note that $X_2(q)\cap X_5(q)=(X_1(q)\cap X_2(q))\cap (X_1(q)\cap X_5(q))$, and we have already identified $X_1(q)\cap X_2(q)$ and 
 $X_1(q)\cap X_5(q)$.
 We may also eliminate relations $R_{45}$.  Relations $R_{45}$ identify $X_4(q)\cap X_5(q)$ which is a direct product of two components: 
 $L_{01}=(X_1(q)\cap X_4(q))\cap (X_1(q)\cap X_5(q))$ and $L_4=(X_3(q)\cap X_4(q))\cap (X_3(q)\cap X_5(q))$, and we have already identified those.

Notice that $|R_4^*\cap R_5^*|=1$ and so $|R_4^*\cup R_5^*|=7$.
Thus we have obtained a presentation 
$$\sigma_G^{(5)}=\langle (D_1\setminus D_X)\cup D_2\cup (D_3\setminus D_Y)\mid (R_1\setminus R_X)\cup R_2\cup (R_3\setminus R_Y)\cup R^*_4\cup R^*_5\rangle$$
with $(5-3)+5+(4-3)=8$ generators and $(20-9)+27+(14-9)+7=50$ relations if $q$ is odd. If $q$ is even the corresponding calculation gives
$(5-3)+6+(4-3)=9$ generators and $(16-5)+20+(10-5)+7=43$ relations.

\subsection{$\widetilde{G}_2 (q)$}

We now use Corollary~\ref{cor::main}. Take $\Delta_1$ based on $a_0$ and $a_1$, $\Delta_2$ based on $a_1$ and $a_2$, and $\Delta_3$ based on $a_0$ and $a_2$.

\begin{center}
\includegraphics{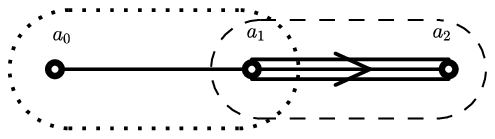}
\end{center}

Then $X_1(q)\cong \SL(3,q)$  has a presentation $\sigma_{X_1(q)}=\sigma_2$,   $X_2(q)\cong G_2(q)$ has a  presentation $\sigma_{X_2(q)}=\sigma_{14}$. Take $X=L_1\cong \SL(2,q)$. Then $X\leq X_i(q)$ for $i=1,2$. Since $X\leq X_2(q)$ and $\sigma_X\subseteq \sigma_{X_1(q)}$, we may use Tietze transformations to remove $D_X$, $R_X$ and $R_{12}$, thus obtaining 
$$\sigma_G=\langle ( D_1 \setminus D_X )\cup D_2\mid  (R_1 \setminus R_X ) \cup R_2\cup R_3^*\rangle.$$
Since $X_3(q)\cong \SL(2,q)\times \SL(2,q)$ and $\SL(2,q)$ is $2$-generated (Proposition~\ref{prop::2_gen}),  we obtain $|R^*_3|=4$, and so $\sigma_G$ has
$(4-3) +6=7$ generators and  $(14 -9) +31 +4=40$ relations if $q$ is odd. If $q$ is even we get $7$ generators and
$(10 -5) + 23 +4 = 32$ relations.

\section{Twisted affine Kac-Moody groups}
\label{sc5}
We briefly  go through the calculations for the remaining 2 infinite series
and 3 exceptional types of twisted affine Kac-Moody groups.

\subsection{$\widetilde{B}_n^t (q)$}

\begin{center}
\includegraphics{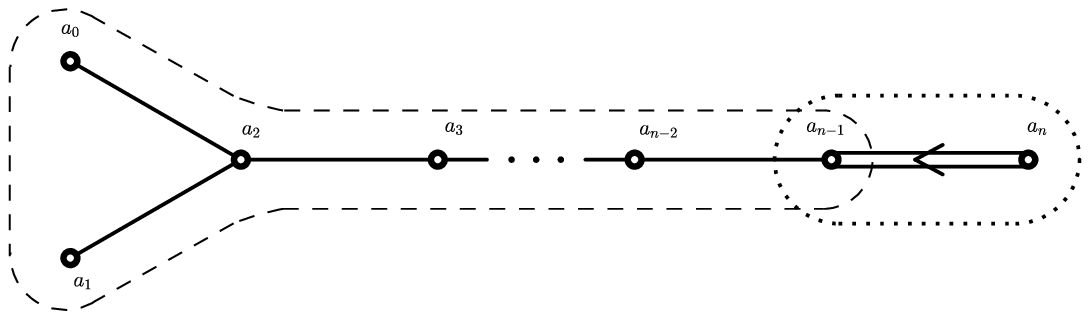}
\end{center}

For $n\neq 3$, the proof  is line by line repetition of  the case $G=\widetilde{B}_n(q)$  giving us the same result:  
 a presentation of $G$ 
with $9$ generators and $56$ relations when $q$ is odd, and with $10$ generators and $49$ relations when $q$ is even,
for $n\geq 9$, and  the same results  as for $\widetilde{B}_n(q)$ for $4\leq n\leq 8$.

Now for $G=\widetilde{B}^t_3(q)$, we repeat the proof of the case $\widetilde{B}_3(q)$ line by line with one change: in the case when $q$ is odd,
we take $\sigma_{X_2(q)}=\sigma_{9}$, thus obtaining a presentation of $G$
 with $5+(5-3)=7$ generators and $20+(27-9)+4=42$ relations if $q$ is odd, and $8$ generators and $35$ relations if $q$ is even.

\subsection{$\widetilde{C}_n^t (q)$} To obtain a presentation of $G=\widetilde{C}_n^t (q)$, $n\geq 3$,  we use Corollary~\ref{cor::main}.

\begin{center}
\includegraphics{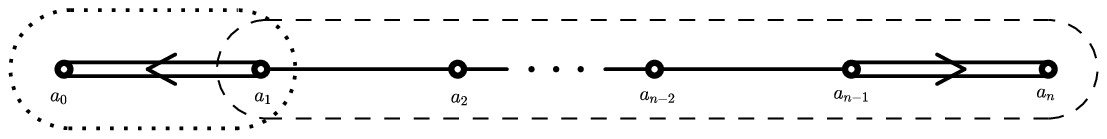}
\end{center}

By Proposition~\ref{prop::sc}, the groups $X_1(q)$, $X_2(q)$ and $X_3(q)$ are simply connected.
Let $\Delta_1=B_n$ be the subdiagram  of $\Delta$ whose  vertices  are the $n$ nodes $a_1,  \ldots, a_n$, and $\Delta_2=C_2$  the subdiagram  of $\Delta$ whose vertices are the nodes $a_0$ and $a_1$. Then $X_1(q)\cong \Spi(2n+1,q)$ and $X_2(q)\cong \Sp(4,q)$. It follows that 
$\Delta_3$  is the subdiagram  of $\Delta$ based on all vertices but $a_1$,  thus of type $A_1\times B_{n-1}$.
Hence, $X_3(q)\cong \SL(2,q)\times \Spi(2n-1,q)$.
Clearly, $\Delta=\Delta_1\cup\Delta_2$. 
Therefore $G$ has a presentation  $$\sigma_G=\langle D_1\cup D_2\mid R_1\cup R_2\cup R_3^*\cup R_{12}\rangle.$$  
as described in Corollary~\ref{cor::main}. 

Take a presentation $\sigma_{X_1(q)}=\sigma_{12}$ if $q$ is odd and $\sigma_{X_1(q)}=\rho_{11}$ if $q$ is even (notice that $B_m(2^a)\cong C_m(2^a)$). 
Consider a subgroup $X=L_1$ of $G$. 
Its Dynkin diagram is of type $A_1$ and so by Proposition~\ref{prop::sc}, $X\cong \SL(2,q)$.
From Table~{3} we know that $X$ has a presentation $\sigma_X=\langle D_X\mid R_X\rangle =\sigma_1$ with $|D_X|=3$ and $|R_X|=9$.
 Now $X\leq X_i(q)$ for $i=1,2$.  
The group $X_2(q)$ has a presentation  $\sigma_{X_2(q)}=\sigma_{10}$.
By Theorem 7.1 of \cite{kn::GKKL3},  $\sigma_X\subseteq\sigma_{X_2(q)}$.
 Since $X\leq X_1(q)$,  obviously,  $D_X\subset X_1(q)$. Thus  elements of $D_X$ can be expressed in terms of elements of $D_1$.
Moreover, the relations  $R_X$ hold,  as they hold in $X_1(q)$.
We use Tietze transformations to eliminate $D_X$,  $R_X$ and $R_{12}$ to obtain: 
 $$\sigma'_G=\langle D_1\cup (D_2\setminus D_X) \mid R_1\cup (R_2\setminus R_X)\cup R_3^*\rangle.$$
 Finally, consider $X_3(q)\cong \SL(2,q)\times \Spi(2n-1,q)$. Each factor  has two  generators (Proposition~\ref{prop::2_gen}). Thus as before we obtain
$|R^*_3|=4$.

 Therefore  $G$ has  a presentation with 
$8+(6-3)=11$ generators and $47+(28-9)+4=70$ relations if $q$ is odd. For even $q$ the corresponding calculation gives
$9+(6-3)=12$ generators and $40+(20-5)+4=59$ relations.
For $3 \leq n \leq 8$ we get shorter presentations, see Table~{1}.

If $n=2$, we use $\sigma_{X_1(q)}=\sigma_9$  if $q$ is odd and $\sigma_{X_1(q)}=\rho_{10}$ if $q$ is even, thus obtaining a presentation
with $5+(6-3)=8$ generators and $27+(28-9)+4=50$ relations if $q$ is odd, and
$6+(6-3)=9$ generators and $20+(20-5)+4=39$ relations if $q$ is even.

\subsection{$\widetilde{C}_n^{\prime} (q)$}
\begin{center}
\includegraphics{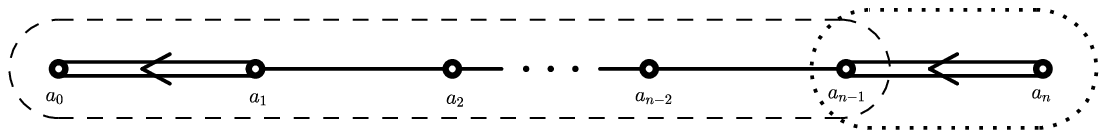}
\end{center}
For $n \geq 3$, the proof is very similar to the case $G=\widetilde{C}_n(q)$. Here we take $\Delta_1 = B_n$ based on the $n$ vertices $a_0, \ldots a_{n-1}$
and $\Delta_2 = C_2$ based on $a_{n-1}$ and $a_n$. Then $\Delta_3$ is based on all vertices but $a_{n-1}$,
thus has type $B_{n-1} \times A_1$. 

$X_2 (q) \cong \Sp (4, q)$ and we again take $\sigma_{X_2(q)} = \sigma_9$. $X_1 (q) \cong \Spi (2n +1, q)$ 
so we take $\sigma_{X_1(q)} = \sigma_{12}$, which has the same size as $\sigma_{11}$. 

The rest of the proof is exactly like the $G=\widetilde{C}_n(q)$ case. We get a presentation with $8+(5-3)=10$ generators
 and $47+(27-9) +4 =69$ relations if $q$ is odd, and 
$9+(6-3)=12$ generators and and $40+(20-5) +4=59$ relations if $q$ is even. For $3 \leq n \leq 8$ we get shorter presentations, see Table~$1$.

For $n=2$ we repeat the proof of the $G=\widetilde{C}_2(q)$ case, but now $\sigma_X \subseteq \sigma_{X_1 (q)}$, 
so our presentation becomes:
 $$\sigma'_G=\langle  (D_1 \setminus D_X) \cup D_2 \mid (R_1 \setminus R_X) \cup R_2\cup R_3^*\rangle.$$
This does not change the calculation, $G$ still has a presentation with $7$ generators and $49$ relations if $q$ is odd
and
 $9$ generators and $39$ relations if $q$ is even.
 

\subsection{$\widetilde{F}_4^t (q)$}

\begin{center}
\includegraphics{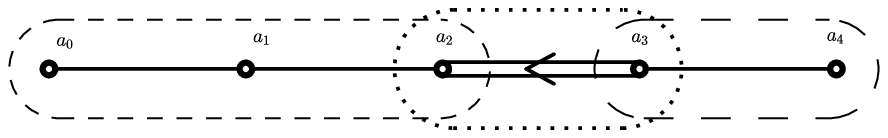}
\end{center}
The proof is line by line repetition of  the case $G=\widetilde{F}_4(q)$. 
The outcome is the same: a presentation with $8$ generators and  $50$ relations if $q$ is odd and $9$ generators and $43$ relations if $q$ is even.

\subsection{$\widetilde{G}_2^t (q)$}

\begin{center}
\includegraphics{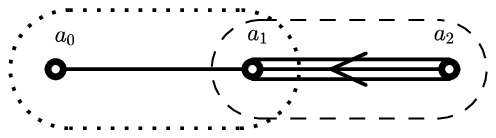}
\end{center}

The argument  follows the proof in the case $G=\widetilde{G}_2(q)$. 
The outcome is  the same: a presentation of $G$ with  $7$ generators and $40$ relations if $q$ is odd and $7$ generators and $32$ relations if $q$ is even.

\section{Adjoint and Classical Groups}
\label{sc6}
So far we have worked with presentations
of a simply connected Kac-Moody group $X(q)$ defined over a finite field $\bF=\bF_q$. 
Our method can be used to derive a presentation of a
Kac-Moody
group that is not necessarily simply connected.
In this section we deal with adjoint and classical groups. 
A reader can use our approach to derive a presentation
of an arbitrary Kac-Moody group of his choosing.

There are two different meanings
of the term {\em adjoint group}.
For a group $X(q)$, its adjoint group is its image under the natural homomorphism
$X(q) \rightarrow \Aut(X(q))$ given by the adjoint action:
 $X(q)_{ad}:=X(q)/Z(X(q))$.

Besides the adjoint group  $X(q)_{ad}$,
there is also a group of points
for an adjoint root datum.
A convenient language to discuss this is a language of 
{\em group $\bF$-functors}, the functors from the category
of commutative $\bF$-algebras to groups.
A Kac-Moody group $G_\mD(\bF)$  
with a root datum $\mD=(I, A, \mX, \mY, \Pi,\Pi^\vee )$ is the result of applying
a group functor $G_\mD$ to the field $\bF$.
The Kac-Moody group $G_\mD(\bF)$ 
is called {\em adjoint}
if $\Pi$ is a basis of $\mX$.
We denote the  adjoint Kac-Moody group $G_\mD(\bF)$
by  $X_{ad}(q)$.

A homomorphism of the root data induces a homomorphism of the group
$\bF$-functors
$\pi:X \rightarrow X_{ad}$.
Taking points over $\bF$ yields 
a group homomorphism 
$\pi (q) :X(q) \rightarrow X_{ad}(q)$.
The kernel of $\pi (q)$ is the centre $Z(X(q))$ of $X(q)$.
Hence, we have an exact sequence of groups 
$$
1\rightarrow Z(X(q)) \rightarrow X(q) \xrightarrow{\pi(q)}
X_{ad}(q)
$$
and $X(q)_{ad}$ is observed in this sequence as the image of $\pi(q)$.
For instance, if $X=A_{n-1}$, it is 
$$
1\rightarrow \mu_n(\bF_q) \rightarrow 
A_{n-1}(q) = \SL_n(q) \xrightarrow{\pi(q)}
(A_{n-1})_{ad}(q) = \mbox{PGL}_n(q)
$$
where $\mu_n$ is the group scheme of the $n$-th roots of unity
and
$A_{n-1}(q)_{ad}=\PSL_n(q)$ is the image of $\pi(q)$ in $\mbox{PGL}_n(q)$.  

Another insightful example is $X=\widetilde{A}_{n-1}$. 
The key exact sequence is
$$
1\rightarrow \mu_n(\bF_q) \times \bF_q^\times  \rightarrow 
\widetilde{A}_{n-1}(q) =\widetilde{\SL}_n(\bF_q[t,t^{-1}])
\xrightarrow{\pi(q)}
(\widetilde{A}_{n-1})_{ad}(q) = \bF_q^\times\ltimes\mbox{PGL}_n(\bF_q[t,t^{-1}])
$$
where 
the simply connected group 
$\widetilde{A}_{n-1}(q)$
is the Steinberg central extension
of $\SL_n(\bF_q[t,t^{-1}])$ by $\bF_q^\times$
and  
the adjoint group 
$(\widetilde{A}_{n-1})_{ad}(q)$
is the semidirect product 
where the action of 
$\bF_q^\times$
is given by
$\alpha \cdot \sum_k P_k t^k= \sum_k \alpha^k P_k
t^k$.


Let $\mP$ be the weight lattice, $\mQ$ the root lattice of the
corresponding Kac-Moody Lie algebra. The weight lattice $\mP$ is
the root lattice $\mX$ for a simply connected root datum.
{\em Mutatis mutandis}, $\mQ$ for an adjoint root datum.
The natural map
$p:\mQ\rightarrow \mP$ is given by the Cartan matrix 
(or its transpose, depending on conventions). It is a part of an
exact sequence
$$
\mQ \xrightarrow{p} \mP \rightarrow  \mZ \rightarrow 1
$$
where $\mZ=\cok p$.
The Cartan matrix
pinpoints all 
the tori (of the corresponding Kac-Moody groups) of interest for us:
$$
Z (q)= Z(X(q)) = \hom (\mZ, \bF_q^\times), \ 
T (q) = \hom (\mP, \bF_q^\times), \ 
T_{ad}(q) = \hom (\mQ, \bF_q^\times), \ 
\pi(q) (\vx) = \vx\circ p.
$$
Let us examine 
the corresponding (not exact) sequence of tori
$$
1\rightarrow Z(X(q)) \rightarrow T(q) \xrightarrow{\pi(q)}
\overline{T}(q)\hookrightarrow T_{ad}(q)
$$
where $\overline{T}(q) = T(q)/Z(X(q))$
can be thought of as a torus of $X(q)_{ad}$.

\begin{prop}
\label{ShExSeq}
Let $X(q)$ be a simply-connected irreducible Kac-Moody group over a finite field $\bF=\bF_q$
(finite, affine or indefinite).
Let
$H(q) := \Ex(\mZ, \bF_q^\times)$ (in the category of abelian groups) 
in the finite or indefinite case 
and
$H(q) := \Ex(\mZ, \bF_q^\times)\times \bF_q^\times$ in the affine case.
Then there exists a short exact sequence 
\begin{equation}
\label{shortKM}
1 
\rightarrow
X(q)_{ad} 
\rightarrow
X_{ad} (q)
\rightarrow
H(q)
\rightarrow
1.
\end{equation}
\end{prop}
\begin{proof}
Let us assume that $X$ is  of  finite or  indefinite type.
Then $p: \mQ \rightarrow \mP$ is injective and $\mZ$ is finite. The long exact
sequence in cohomology
$$
1 \rightarrow \hom (\mZ, \bF_q^\times)
\rightarrow
\hom (\mP, \bF_q^\times)
\xrightarrow{\pi(q)}
\hom (\mQ, \bF_q^\times)
\rightarrow
\Ex (\mZ, \bF_q^\times)
\rightarrow
1
$$
reduces to a short exact sequence connecting
the adjoint tori
\begin{equation}
\label{shortToriKM}
1 
\rightarrow
\overline{T}(q) 
\rightarrow
T_{ad} (q)
\rightarrow
H(q)
\rightarrow
1.
\end{equation}
This implies the existence of 
the short exact sequence~(\ref{shortKM}).

If $X$ is  affine, 
the map $p: \mQ \rightarrow \mP$ is no longer injective.
We can decompose $\mQ=\mQ^\prime \times \bZ$ where
$\bZ = \ker\, p$ and  $p: \mQ^\prime \rightarrow \mP$ is injective.
The long exact
sequence in cohomology is
$$
1 \rightarrow \hom (\mZ, \bF_q^\times)
\rightarrow
\hom (\mP, \bF_q^\times)
\xrightarrow{\pi(q)}
\hom (\mQ^\prime , \bF_q^\times)
\rightarrow
\Ex (\mZ, \bF_q^\times)
\rightarrow
1.
$$
It gives a description of 
the tori using an auxiliary group $T^\prime (q) = \hom (\mQ^\prime ,
\bF_q^\times)$.
The sequence
$$
1 
\rightarrow
\overline{T}(q) 
\rightarrow
T^\prime (q)
\rightarrow
\Ex (\mZ, \bF_q^\times)
\rightarrow
1
$$
is exact.
Since
$T_{ad} (q) = T^\prime (q) \times \bF_q^\times$,
a direct product with $\bF_q^\times$
establishes 
exact sequence~{(\ref{shortToriKM})}
in the affine case.
This 
proves an existence of 
exact sequence~{(\ref{shortKM})}
in all cases.
\end{proof}

Proposition~\ref{ShExSeq}
gives presentations of both $X(q)_{ad}$ and $X_{ad}(q)$.
Since $X(q)_{ad} = X(q)/Z(q)$,
one gets  $X(q)_{ad}$
from
$X(q)$ by ``killing'' generators of  $Z(q)$.
The presentation of $X_{ad}(q)$ is obtained
from presentations of  $X(q)_{ad}$ and $H(q)$ by 
P.~Hall's Lemma \cite[Lemma 2.2]{kn::CapLuRe}.
Observe that the right conjugations in P.~Hall's
Lemma are superfluous. One usually adds them for convenience.

\begin{cor}
\label{Hall}
Suppose we have a presentation of $X(q)$, $Z(q)$ and $H(q)$: 
$$\sigma_{X(q)} = \langle D\mid R \rangle,
\ \ 
\sigma_{Z(q)} = \langle D_1\mid R_1 \rangle,
\ 
\ 
\sigma_{H(q)} = \langle D_2\mid R_2\rangle .
$$
Then we have presentations of adjoint groups
$$\sigma_{X(q)_{ad}} = \langle D\mid R\cup D_1^\sharp \rangle
\ \mbox{and}\  \
\sigma_{X_{ad}(q)} = \langle D\cup D_2\mid R\cup D_1^\sharp\cup 
R_2^\sharp \cup D_2^{act} \rangle
$$
where 
$D_1^\sharp = \{ x^\sharp =1 \mid x\in D_1, x^\sharp\ \mbox{is an expression of } x \mbox{ in } D\}$,\\
$R_2^\sharp = \{ w=w^\sharp \mid w\in R_2, w^\sharp\in  X(q)_{ad}  \mbox{ is an expression of }\ w(D_2)\ \mbox{in}\ D\}$
and\\
$D_2^{act} = \{ xax^{-1} = \,^xa(D) \mid  x\in D_2, a \mbox{ is a generator of } X(q)_{ad},  \,^xa(D)\mbox{ is an expression of}\ xax^{-1}\ \mbox{in}\ D\}$.
\end{cor}

The group $\mP/\mQ$ is computed by calculating
the integral Smith normal forms of Cartan matrices.
We summarise  these calculations in Table~{4}. 
We list Dynkin's labels  for affine groups and Cartan labels
for finite groups in the first column.
The second column contains the group $\mP/\mQ$: by 
$(a_1,\ldots  , a_k)$ we mean the group $\bZ/a_1 \times \ldots \times \bZ/a_k$.

The next two columns are related to $X(q)_{ad}$.
The third column lists the minimal number of generators for $Z(q)$.
We get a generator, if there is a non-trivial
homomorphism from a cyclic direct summand $\bZ/k$
of $\mP/\mQ$ to $\bF_q^\times$.
We introduce a symbol $\fA(k)$, equal to 1, 
if $gcd(k,q-1)>1$
and 0 if $gcd(k,q-1)=1$.
The fourth column is a maximal  possible value of $\fA(k)$ taken over all $q$:
this is the number of extra relations
to describe $X(q)_{ad}$ for generic $q$.

The right three columns are related to $X_{ad}(q)$.
The third column lists the 
minimal number of generators for $H(q)$.
We get a generator
if there is a non-trivial
quotient by the $k$-th powers, 
where $\bZ/k$ is a  direct summand of $\mP/\mQ$:
$\Ex (\bZ/k, \bF_q^\times) \cong \bF_q^\times/ (\bF_q^\times)^k$.
This is controlled by the symbol $\fA(k)$.
No generator arises from the infinite cyclic group:
$\Ex (\bZ, \bF_q^\times) =0$, yet
the infinite cyclic group appears only in the affine types
where $H(q)$ hass an extra generator. Hence, $|D_1|=|D_2|$.
The fourth column uses a  maximal  possible value of $\fA(k)$: 
this is a number of extra generators needed to describe
$X_{ad}(q)$ for generic $q$.
The last column is the maximal cardinality of
$D_1^\sharp\cup R_2^\sharp \cup D_2^{act}$,
the number of extra relations needed to describe $X_{ad}(q)$.
In our computation we use the estimates
$|D_1^\sharp|=|D_1|=|D_2|$,
$|R_2^\sharp|=  |R_2| = |D_2|$
and
$|D_2^{act}| = 2|D_2|$.
The latter holds because $X(q)$ is generated by 2
elements (with few exceptions, see Theorem~\ref{theorem::2_gen}).
Hence, $|D_1^\sharp \cup R_2^\sharp \cup D_2^{act}| = 4|D_1|$.

\begin{table}
\label{ta4}
\begin{center}
\caption{Extra Generators and Relations for  $\widetilde{X}(q)_{ad}$
  and $\widetilde{X}_{ad}(q)$}
\vskip 3mm
\bgroup
\def\arraystretch{1.3}
\begin{tabular}{|c|c||c|c|c|} \hline
  $X_n$ & $\mP/\mQ$ &  $|D_1|=|D_2|$ & max $(|D_1|= |D_2|)$ &
  max  $|D_1^\sharp \cup R_2^\sharp \cup D_2^{act}|$ \\ 
\hline
$A_n$ & $(n+1)$ & $\fA(n+1)$   & 1  & 4 \\
$B_n,C_n,E_7$ & $(2)$ &  $\fA(2)$  & 1 & 4\\
$D_{2n}$ & $(2,2)$ & $2\fA(2)$    & 2  & 8 \\
$D_{2n+1}$ & $(4)$ & $\fA(2)$   & 1  & 4 \\
$G_2,F_4,E_8$ & $()$ & 0 & 0 & 0 \\
$E_{6}$ & $(3)$ & $\fA(3)$ &  1  &  4 \\
$\widetilde{A}_{n-1}$ & $(0,n)$ & $1+\fA(n)$  & 2
  & 8 \\
$\widetilde{B}_n,\widetilde{C}_n,\widetilde{E}_7,\widetilde{B}^t_n,\widetilde{C}^t_n$
  & $(0,2)$ & $1+\fA(2)$   & 2 & 8 \\
$\widetilde{D}_{2n}$ & $(0,2,2)$ & $1+2\fA(2)$  & 3 & 12 \\
$\widetilde{D}_{2n+1}$ & $(0,4)$ & $1+\fA(2)$ & 2 &
  8 \\
$\widetilde{G}_2,\widetilde{F}_4,\widetilde{E}_8,\widetilde{C}^{\prime}_n,\widetilde{F}^t_4,\widetilde{G}^t_2$
  & $(0)$ & 1 & 1  & 4 \\
$\widetilde{E}_{6}$ & $(0,3)$ & $1+\fA(3)$  & 2 &  8 \\
\hline
\end{tabular}
\egroup
\vskip 3mm
\end{center}
\end{table}

As an application of our techniques we write down 
the numbers of generators and
relations of the remaining classical groups over 
$\bF_q[t, t^{-1}]$ in Table 5 (for sufficiently large $q$). 
The groups $\SL_n$, $\Spi_n$ and $\Sp_{2n}$ are
simply connected, so they are already in Tables 1 and 2. The group
$$
\PSL_n(\bF_q[t, t^{-1}]) = \widetilde{A}_{n-1}(q)_{ad}
$$
is adjoint, hence its presentation follows from Tables 1 and 4. The groups
$$
\PGL_n(\bF_q[t, t^{-1}]) \lhd (\widetilde{A}_{n-1})_{ad} (q), \ \ \
\SO_{2n+1}(\bF_q[t, t^{-1}]) \lhd  (\widetilde{B}_{n})_{ad}(q) 
$$
are normal subgroups in the adjoint groups 
(before the semidirect product). Similarly to 
Proposition~\ref{ShExSeq} they appear in an exact sequence
$$1 
\rightarrow
X(q)_{ad} 
\rightarrow
{\bf G} (\bF_q[t, t^{-1}])
\rightarrow
\Ex(\mZ, \bF_q^\times)
\rightarrow
1,$$
hence they get a presentation
as in Corollary~\ref{Hall} but with $\Ex (\mZ, \bF_q^\times)$ instead
of $H(q)$.

Finally, $\SO_{2n}$ is not related to the adjoint group. It is an
intermediate quotient fitting into exact sequence of group schemes
$$1 
\rightarrow
\bZ/2 
\rightarrow
\Spi_{2n}
\rightarrow
\SO_{2n}
\rightarrow
1.$$
Using our arguments, we fit the group into an exact sequence 
$$1 
\rightarrow
\widetilde{D}_{n}(q) / Z 
\rightarrow
\SO_{2n} (\bF_q[t, t^{-1}])
\rightarrow
\Ex(\bZ/2, \bF_q^\times)
\rightarrow
1$$
where the central subgroup $Z$ is isomorphic to $\hom(\bZ/2,
\bF_q^\times)$.
We get a presentation
as in Corollary~\ref{Hall} where the result depends on whether $q$ is
even or odd.

\begin{table}
\label{ta4a}
\begin{center}
\caption{Generators and Relations of Classical $\mathbf{G}(\bF_q[t, t^{-1}])$}
\vskip 3mm
\bgroup
\def\arraystretch{1.3}
\begin{tabular}{|c|c|c|c||c|c|c|c|c|} \hline
${\bf G}$ & $D_{\sigma}$ & $R_{\sigma}$, $q$ &  $R_{\sigma}$, $q$ & ${\bf G}$ & $D_{\sigma}$ &
$R_{\sigma}$ & $D_{\sigma}$ & $R_{\sigma}$   \\ \cline{6-9}
 &  &  {odd} & {even} & & \multicolumn{2}{|c|}{$q$ odd} & \multicolumn{2}{|c|}{$q$ even} \\
\hline
$\PSL_3$ & 5
& 28 
& 24 
& 
$\SO_7$ & 9 & 47 
& 8 & 36 
\\
$\PSL_n$, $4 \leq n \leq 8$ & 7
& 37 
& 33 
& 
$\SO_{9}$ & 9 & 56 
& 9 & 45 
\\ 
$\PSL_n$, $n \geq 9$ & 9
& 45 
& 41 
&
$\SO_{2n+1}$, $5 \leq n \leq 8$ & 9 & 57 
& 9 & 46 
\\ \cline{1-4}
$\PGL_3$ & 6 
& 31 
& 27 
&
$\SO_{2n+1}$, $n \geq 9$ & 10 & 61  
& 10 & 50 
\\ \cline{5-9}
$\PGL_n$, $4 \leq n \leq 8$ & 8 
& 40 
& 36 
&
$\SO_8$ or $\SO_{2n}$, $6\leq n \leq 8$ & 8 & 43 
& 7  & 35 
\\
$\PGL_n$, $n \geq 9$ & 10 
& 48 
& 44 
& 
$\SO_{10}$ & 8 & 44 
& 7 & 36  
\\
& & & & $\SO_{2n}$, $n \geq 9$ & 9 & 47 
& 8 & 39 
\\
\hline
\end{tabular}
\egroup
\vskip 3mm
\end{center}
\end{table}


\begin{thebibliography}{GLS99999}

\bibitem[AM]{kn::AM} P. Abramenko, B. M\"{u}hlherr, 
\emph{Presentations de certaines BN -paires jumelees comme sommes
amalgamee}, 
C. R. Acad. Sci. Paris Ser. I Math.  
 \textbf{325}  (1997), 701--706. 

\bibitem[AsG]{kn::AG} M.~Aschbacher, R.~Guralnick. 
\emph{Some applications of the first cohomology group}, 
J. Algebra \textbf{90} (1984), 446--460.
 
\bibitem[C1]{kn::Cap} I. Capdeboscq, \emph{Bounded presentations of
  Kac-Moody groups}, J. Group Theory \textbf{16} (2013), 899--905.

\bibitem[C2]{kn::Cap2} I. Capdeboscq, 
\emph{On the generation of discrete and topological Kac-Moody groups},
C. R. Math. Acad. Sci. Paris  
\textbf{353}  (2015),  695--699. 

\bibitem[CLR]{kn::CapLuRe}
I. Capdeboscq, A. Lubotzky, B. R\'{e}my, 
{\em Presentations: from Kac-Moody groups to profinite and back},
Transform. Groups
\textbf{21} (2016), 929--951.

\bibitem[CR]{kn::CR}
I. Capdeboscq,  B. R\'{e}my, 
{\em Generation of non-Archimedean reductive groups},
preprint.


\bibitem[Ca]{Ca} P.-E. Caprace, 
\emph{On 2-spherical Kac-Moody groups and their central extensions},
Forum Math.  \textbf{19}  (2007),  763--781.

\bibitem[Car1]{kn::Car} R.~Carter, 
\emph{Kac-Moody groups and their 
automorphisms},   Groups, combinatorics and geometry (Durham, 1990),  
218--228, LMS Lecture Note Ser., \textbf{165}, Cambridge 
Univ. Press, Cambridge, 1992.

\bibitem[Car2]{kn::Car2}  R.~Carter,  
\emph{Lie algebras of finite and affine type}, 
Cambridge Studies in Advanced Mathematics, \textbf{96},  
Cambridge Univ. Press, Cambridge, 2005.


\bibitem[GK]{kn::GK} R.~Guralnick, W.~Kantor,  
\emph{Probabilistic generation of finite simple groups}, 
J. Algebra $\bf{234}$ (2000),  743--792.


\bibitem[GKKaL1]{kn::GKKL1}
R.~Guralnick,  W.~Kantor, M.~Kassabov and A.~Lubotzky,  
\emph{Presentations of finite simple groups: profinite and
  cohomological approaches}, 
Groups Geom. Dyn. {\bf 1} (2007), 469--523.

\bibitem[GKKaL2]{kn::GKKL2}
R.~Guralnick,  W.~Kantor, M.~Kassabov and A.~Lubotzky,  
\emph{Presentations of
finite simple groups: A quantitative approach},
J. Amer. Math. Soc. {\bf 21} (2008), 711--774.


\bibitem[GKKaL3]{kn::GKKL3}
R.~Guralnick,  W.~Kantor, M.~Kassabov and A.~Lubotzky,  
\emph{Presentations of
finite simple groups: A computational approach}, 
J. Eur. Math. Soc. {\bf 13} (2011), 391--458.

\bibitem[KL]{kn::KL} W.~Kantor, A.~Lubotzky,
\emph{The probability of generating a finite classical group}, 
Geom. Dedicata {\bf 36} (1990), 67--87.

\bibitem[MT]{kn::MT} A.~Mar\'{o}ti,  M.~Tamburini Bellani, 
\emph{A solution to a problem of Wiegold}, 
Comm. Algebra {\bf 41} (2013), 34--49.

%
\bibitem[MoRe]{kn::MR} 
J. Morita, U. Rehmann, 
\emph{Symplectic K2 of Laurent polynomials, associated Kac-Moody
  groups and Witt rings}, 
Math. Z.  $\bf{206}$ (1991), 57--66. 



\bibitem[T]{kn::T} J.~Tits,
\emph{Uniqueness and presentation of 
Kac-Moody groups over fields},
Journal of Algebra {\bf 105} (1987), 542--573.


\end{thebibliography}
\end{document}